\documentclass[12pt,oneside]{article}
\usepackage{amssymb,amsmath,amsthm}
\usepackage[colorlinks,pdfusetitle]{hyperref}

\usepackage[T1]{fontenc}
\usepackage[utf8]{inputenc}

\usepackage{bm}

\usepackage{wrapfig}
\usepackage{float}

\usepackage{upgreek}

\usepackage{soul}

\usepackage{tikz}
\usetikzlibrary{calc}
\usetikzlibrary{angles,quotes}
\usetikzlibrary{snakes}

\def\dotMarkRightAngle[size=#1](#2,#3,#4){%
 \draw ($(#3)!#1!(#2)$) -- 
       ($($(#3)!#1!(#2)$)!#1!90:(#2)$) --
       ($(#3)!#1!(#4)$);
}

\usepackage{geometry}
\newcommand\blfootnote[1]{%
  \begingroup
  \renewcommand\thefootnote{}\footnote{#1}%
  \addtocounter{footnote}{-1}%
  \endgroup
}

\newcommand{\VM}[1]{\textbf{%
  \textsl{%
  [\textcolor{olive}{VLM: #1}]%
  }}}    
 \newcommand{\VA}[1]{ \textbf{\textsl{\color{red}[VA: #1]}}}    

\newcommand{\im}{\ensuremath{\mathrm{Im}}}   

\newcommand{\re}{\ensuremath{\mathrm{Re}}}  

\DeclareMathOperator{\seq}{\subseteq}  

\newcommand{\restrict}[1]{\ensuremath{\!\!\upharpoonright_{#1}}}

\newcommand{\N}{\ensuremath{\mathbb{N}}}
\newcommand{\Z}{\ensuremath{\mathbb{Z}}}

\newcommand{\R}{\ensuremath{\mathbb{R}}}
\newcommand{\C}{\ensuremath{\mathbb{C}}}

\renewcommand{\P}{\ensuremath{\mathbb{P}}}  

\newcommand{\gm}{\ensuremath{\mathbb{G}_\mathrm{m}}}  

\renewcommand{\phi}{\varphi}
\renewcommand{\le}{\ensuremath{\leqslant}}

\newcommand{\subs}{\subseteq} 

\newcommand{\nstrong}{\ensuremath{\not\kern-4pt\lhd\;}} 

\newcommand{\cross}{\ensuremath{\times}}

\newbox\noforkbox \newdimen\forklinewidth
\forklinewidth=0.3pt \setbox0\hbox{$\textstyle\smile$}
\setbox1\hbox to \wd0{\hfil\vrule width \forklinewidth depth-2pt
  height 10pt \hfil}
\wd1=0 cm \setbox\noforkbox\hbox{\lower 2pt\box1\lower
2pt\box0\relax}
\def\unionstick{\mathop{\copy\noforkbox}\limits}

\def\nonfork_#1{\unionstick_{\textstyle #1}}

\setbox0\hbox{$\textstyle\smile$} \setbox1\hbox to \wd0{\hfil{\sl
/\/}\hfil} \setbox2\hbox to \wd0{\hfil\vrule height 10pt depth
-2pt width
                \forklinewidth\hfil}
\newbox\doesforkbox
\setbox\doesforkbox\hbox{\lower 2pt\box1 \lower
2pt\box2\lower2pt\box0\relax}
\def\nunionstick{\mathop{\copy\doesforkbox}\limits}

\def\fork_#1{\nunionstick_{\textstyle #1}}

\newcommand{\leteq}{\mathrel{\mathop:}=}

\newcommand{\dd}{\ensuremath{\mathrm{d}}}

\newcommand{\Vbar}{\overline{V}}

\theoremstyle{plain}
\newtheorem{theorem}{Theorem}[section]
\newtheorem{lemma}[theorem]{Lemma}
\newtheorem{prop}[theorem]{Proposition}
\newtheorem{proposition}[theorem]{Proposition}
\newtheorem{conjecture}[theorem]{Conjecture}
\newtheorem{fact}[theorem]{Fact}
\newtheorem{example}[theorem]{Example}
\newtheorem{corollary}[theorem]{Corollary}

\theoremstyle{definition}

\newtheorem{notation}[theorem]{Notation}

\newtheorem{remark}[theorem]{Remark}

\newtheorem*{claim*}{Claim}

\newcommand{\be}{\begin{equation}}
\newcommand{\ee}{\end{equation}}

\newcommand{\lambar}{\ensuremath{\boldsymbol\lambda}}

\newcommand{\albar}{\ensuremath{\alpha}}

\newcommand{\cbar}{\ensuremath{\boldsymbol{c}}}

\newcommand{\wbar}{{\ensuremath{\boldsymbol{w}}}}
\newcommand{\zbar}{{\ensuremath{\boldsymbol{z}}}}
\newcommand{\xbar}{{\ensuremath{\boldsymbol{x}}}}
\newcommand{\abar}{{\ensuremath{\boldsymbol{a}}}}

\newcommand{\ybar}{{\ensuremath{\boldsymbol{y}}}}

\newcommand{\gambar}{{\ensuremath{\boldsymbol{\gamma}}}}
\newcommand{\alphabar}{{\ensuremath{\boldsymbol{\upalpha}}}}

\newcommand{\Log}{\mathrm{Log}}
\newcommand{\Sbar}{{\mathbf{S}}}
\newcommand{\Fbar}{{\mathbf{F}}}
\newcommand{\Gbar}{{\mathbf{G}}}

\newcommand{\Dte}{\ensuremath{D_{(\theta,\eta)}}}

\renewcommand{\leq}{\leqslant}
\renewcommand{\geq}{\geqslant}

\providecommand{\keywords}[1]
{
    \noindent
  {\small	
  \textbf{\textit{Keywords.}} #1}
}

\providecommand{\subjclass}[1]
{
    \noindent
  {\small	
  \textbf{\textit{2020 MSC.}} #1}
}

\title{A geometric approach to some systems of exponential equations}

\author{\texorpdfstring{Vahagn Aslanyan\textsuperscript{1}, Jonathan Kirby\textsuperscript{2}, and Vincenzo Mantova\textsuperscript{3}  \\
        \small \textsuperscript{1,2}University of East Anglia, Norwich, UK \\
        \small \textsuperscript{3}University of Leeds, Leeds, UK \\}{Vahagn Aslanyan, Jonathan Kirby, and Vincenzo Mantova}
}

\date{28 September 2021}

\begin{document}

\maketitle

\begin{abstract}
  Zilber’s Exponential Algebraic Closedness conjecture (also known as Zilber’s Nullstellensatz) gives conditions under which a complex algebraic variety should intersect the graph of the exponential map of a semiabelian variety.

  We prove the special case of the conjecture where the variety has dominant projection to the domain of the exponential map, for abelian varieties and for algebraic tori. Furthermore, in the situation where the intersection is 0-dimensional, we exhibit structure in the intersection by parametrizing the sufficiently large points as the images of the period lattice under a (multivalued) analytic map. Our approach is complex geometric, in contrast to a real analytic proof given by Brownawell and Masser just for the case of algebraic tori. 
\end{abstract}

\keywords {Complex exponentiation, Exponential Algebraic Closedness, Zilber's conjecture.}

\subjclass {Primary: 11G10, 14K20; Secondary: 03C60, 12L12}

\blfootnote{VA and JK were supported by EPSRC grant EP/S017313/1. VM was supported by EPSRC grant EP/T018461/1.}

\tableofcontents

\section{Introduction}

In his model-theoretic study of the complex exponential function, Zilber \cite{Zilb-pseudoexp} asked what systems of equations built from polynomials and the exponential function have solutions in the complex field. The analogous question just for polynomials is solved by the fundamental theorem of algebra and the Hilbert Nullstellensatz.

The fact that the exponential map is a homomorphism of algebraic groups places some restrictions. For example, the following system of equations does not have a solution in $\C$:
\begin{equation*}
\begin{cases}
2z_1 = z_2+1,\\
(e^{z_1})^2 = e^{z_2},
\end{cases}\,
\end{equation*}
because $(e^{z_1})^2 = e^{2z_1}$, and so having a solution would imply $e^1=1$.

Further strong restrictions are predicted by Schanuel's conjecture of transcendental number theory \cite[p30]{Lang-tr}, which asserts $\operatorname{tr.deg.}_{\mathbb{Q}}(z_1,\dots,z_n,e^{z_1},\dots,e^{z_n}) \geq n$ for any complex numbers $z_1,\dots,z_n$ that are linearly independent over $\mathbb{Q}$.
For example, a simple application of the conjecture gives the algebraic independence of $e$ and $\pi$ (an open question), so it would follow that for any non-zero rational polynomial $p(z,w)$, there is no solution to the system of equations
\begin{equation*}
e^z = -1 , \qquad  p(z,e^1) = 0.
\end{equation*}

Zilber formulated a precise conjecture that captures the idea that every system of equations should have a solution unless that would contradict Schanuel's conjecture, which we call his \emph{Exponential Algebraic Closedness} conjecture, or \emph{EAC} conjecture (sometimes also called \emph{Zilber's Nullstellensatz} \cite{DMT16}).

The EAC conjecture is expressed in geometric terms. Let $\gm^n$ be the algebraic torus of dimension $n$. Since we are exclusively working over $\C$, we shall identify $\gm$ with its complex points, so  $\gm = \gm(\C) = \C^\times$. 
\begin{conjecture}[EAC \cite{Zilb-pseudoexp}]\label{original EC}
Let $V \subs \C^n \cross \gm^n$ be a free and rotund variety. Then there is a point $\zbar \in \C^n$ such that $(\zbar, \exp(\zbar)) \in V$.
\end{conjecture}
Here $\zbar = (z_1,\ldots,z_n)$ and $\exp(\zbar)$ means the tuple $(e^{z_1},\ldots,e^{z_n})$.
The \emph{freeness} property in Conjecture~\ref{original EC} is related to the constraints from the exponential map being a group homomorphism, and \emph{rotundity} is related to the constraints from Schanuel's conjecture. 
We shall omit the precise definitions of these properties, as they are slightly technical and we will not need them.

If both the Schanuel and the EAC conjectures are true, they would give a complete characterisation of the systems which have solutions. Zilber also showed in \cite{Zilb-pseudoexp} (see also \cite{Bays-Kirby-2013}) that, together, Schanuel's conjecture and a stronger form of EAC imply strong consequences for the model-theoretic structure $(\mathbb{C}, +, \times, \exp)$, in particular that it is quasiminimal, a previous conjecture by Zilber \cite{Zilb-gen-anal-sets} that is still open to this day.

The number-theoretic part of Schanuel's conjecture seems out of reach.  However, the functional part of Schanuel's conjecture was proved by Ax \cite{Ax} and that implies that it is generically true \cite[Theorem~1.4]{Kirby-Schanuel}. In particular, a positive solution to Zilber's EAC conjecture would characterise the systems for which the existence of solutions is essentially a number-theoretic transcendence problem, rather than a functional transcendence or geometric problem. Moreover, we now know that EAC directly implies the quasiminimality of $(\mathbb{C}, +, \times, \exp)$ \cite[Theorem~1.5]{Bays-Kirby-exp}.

Apart from the classical exponential function, one can consider other periodic functions such as the Weierstrass $\wp$-functions and and their derivatives. For example, in Section~\ref{example section} of the paper we will describe solutions to the following equation.
\begin{example}\label{example prop}
Let $\wp$ be any Weierstrass $\wp$-function.
Then there are $z \in \C$ such that $\wp'(\wp(z)^2) = z$, and indeed we can find 12 infinite families of solutions parametrised by the pairs $(\omega_1,\omega_2) \in \Lambda^2$, for $|\omega_1|, |\omega_2|$ sufficiently large, where $\Lambda$ is the period lattice of $\wp$.
\end{example}
The $\wp$-functions are essentially the exponential maps of elliptic curves. Our method for $\wp$ also applies to a wide range of systems of equations, dealing with the exponential maps of abelian or semiabelian varieties.

In this generality the EAC conjecture becomes the following.
\begin{conjecture}[EAC for semiabelian varieties]\label{EAC abelian conjecture}
Let $S$ be a complex semiabelian variety of dimension $n$, and write $\exp_S : \C^n \to S$ for its exponential map. Let $V \subs \C^n \cross S$ be a {free} and {rotund} subvariety. Then there is $\zbar \in \C^n$ such that $(\zbar,\exp_S(\zbar)) \in V$.
\end{conjecture}

Again, the notions of freeness and rotundity relate to $\exp_S$ being a homomorphism and to the semiabelian version of Schanuel's conjecture. We will not need them in this paper. We refer the interested reader to \cite[Definition~2.26]{Kirby-semiab} or \cite[Definition~7.1]{Bays-Kirby-exp} for more details. At least when $S$ is simple, EAC for $S$ also implies that the structure $(\C, +, \times, \exp_S)$ is quasiminimal \cite[Theorem~1.9]{Bays-Kirby-exp}.

As a notational convention, since all the algebraic varieties we will consider will be defined over $\C$, we will identify them with their sets of $\C$-points. For example, above we write $S$ and $V$ rather than $S(\C)$ and $V(\C)$. We will also write $\P_n$ for complex projective $n$-space rather than $\P_n(\C)$.
We denote points in affine and projective spaces by boldface letters such as $\zbar$, and their coordinates by standard letters with subscripts such as $z_1$. 
We say that a point in $V$ of the form $(\zbar,\exp_S(\zbar))$ is an \emph{exponential point of $V$}.

In this paper we establish the following family of instances of Conjecture \ref{EAC abelian conjecture} in the case of abelian varieties.
\begin{theorem}\label{dominant abelian case}
Let $A$ be a complex abelian variety of dimension $n$. Let $V \subs \C^n \cross A$ be an algebraic subvariety with dominant projection to $\C^n$, that is, its projection to $\C^n$ has dimension $n$. Then there is $\zbar \in \C^n$ such that $(\zbar, \exp_{A}(\zbar)) \in V$. 

\end{theorem}
It will follow easily from our proof that the set $\{ \zbar \in \C^n: (\zbar, \exp_{A}(\zbar)) \in V \}$ is in fact Zariski dense in $\C^n$, and actually that the points $(\zbar, \exp_A(\zbar))$ are Zariski dense in $V$. Moreover, we show that almost all of the large solutions are parametrised in terms of the period lattice $\Lambda$ of $\exp_A$. (See Theorem~\ref{thm-main-abelian} for the details.)

A subvariety with dominant projection as in this statement is automatically rotund, and can be easily reduced to a free and rotund subvariety. Hence, Theorem \ref{dominant abelian case} is indeed a special case of Conjecture \ref{EAC abelian conjecture}.

The analogous theorem for algebraic tori was proven by Brownawell and Masser.

\begin{theorem}[{\cite[Prop.\ 2]{brown-masser}}]
\label{algebraic torus case}
Let $V \subs \C^n \cross \gm^n$ be an algebraic subvariety with dominant projection to $\C^n$. Then there is a point $\zbar \in \C^n$ such that $(\zbar, \exp(\zbar)) \in V$. 
\end{theorem}

They used Newton's iterative method to approximate solutions, and in particular Kantorovich's theorem which gives criteria for these approximations to converge to an actual solution.
Another account of the same proof is given in \cite{daquino--fornasiero-terzo}.
A similar theorem for the modular $j$-function was established in \cite{eterovic-herrero} using Rouch\'e's theorem of complex analysis in place of Kantorovich's theorem.
Using our methods we also give a new proof of Theorem~\ref{algebraic torus case}.

Unlike \cite{brown-masser} and \cite{daquino--fornasiero-terzo}, in our proofs of Theorems~\ref{dominant abelian case} and~\ref{algebraic torus case} we exploit the geometry and topology of the system as much as possible; we do not use Kantorovich's theorem or Rouch\'e's theorem.

To explain our approach it is easiest to go back to the idea of the proof of Brownawell and Masser. If $(\zbar,\wbar) \in V$ with $\zbar \in \C^n$ and $\wbar \in \gm^n$ then the dominant projection assumption means that generically we can regard $\wbar$ as $\alpha(\zbar)$, where $\alpha$ is an algebraic map. The problem reduces to finding a zero of
 \[F(\zbar) \leteq \exp(\zbar) - \alpha(\zbar).\]
They consider $\zbar \to \infty$ in a small complex neighbourhood of a real straight line and prove (after some rescaling) that when $\zbar = \lambar + \Log(\alpha(\lambar))$ with $\lambar \in (2\pi i\Z)^n$ then $F(\zbar)$ is small enough that Newton's method will converge to a zero of $F$ near it.
They then observe that the solutions they find are indexed by ``sufficiently many'' lattice points to give a Zariski dense set of solutions.

In the general abelian case, we face two technical challenges. The exponential and logarithmic maps of $\gm$ are very well understood, and are easy to differentiate explicitly to perform the necessary computations. This is less practical for abelian varieties. Moreover, whereas in the $\gm$ setting both $\exp(\zbar)$ and $\alpha(\zbar)$ lie in $\gm^n \subs \C^n$, in the abelian setting both quantities lie in $A$, which is a projective variety.

We choose a convenient affine chart by taking the logarithm, so we work in the covering space $\C^n$, and we define a new map $F$ (at least locally) as
\[F(\zbar) \leteq \zbar - \Log_A(\alpha(\zbar)).\]
Instead of looking for zeros of $F$, we now want to find $\zbar \in \C^n$ such that $\exp_A(F(\zbar)) = 0_A$, or equivalently, $F(\zbar) \in \Lambda$, the period lattice of $\exp_A$. Essentially from the compactness of $A$, we show that the second term $\Log_A(\alpha(\zbar))$ is convergent as $\zbar \to \infty$ along most real lines, and so $F$ is asymptotically a translation and hence is locally invertible. Writing $S$ for the local inverse, the (sufficiently large) points $\zbar$ such that $\alpha(\zbar) = \exp_A(\zbar)$ are then $S(\lambar)$ for $\lambar \in \Lambda$. So as well as finding solutions, we give an analytic parametrisation of them by lattice points.

The algebraic map $\alpha$, the logarithm $\Log_A$, and the maps $F$ and $S$ are multivalued maps, so the above argument will be done locally around points at infinity, after embedding $\C^n$ into $\P_n$ in the usual way. By considering all the branches of $\alpha$ and $\Log_A$, we can then parametrise the points $(\zbar, \exp_A(\zbar)) \in V$ locally via the corresponding branches of $S$.

To summarise how $S$ parametrises the solutions globally, we shall ultimately describe $S$ as a sheaf of analytic functions (in the sense of Remark~\ref{sheaf remark}) on an open subset $\Omega^*$ of $\P_n$ , which is `large' in the sense that it contains a Zariski open dense subset of the hyperplane at infinity. We shall verify that every $\zbar \in \Omega^* \cap \C^n$ such that $(\zbar, \exp_A(\zbar)) \in V$ is of the form $\zbar = S(\lambar)$ for some $\lambar \in \Lambda$. 

Our method also works in the algebraic torus case. Since $\gm$ is not compact, a little more analysis of growth rates is needed, although still less than in the Brownawell-Masser proof. In some final remarks at the end of the paper we discuss how far this method might be pushed.

The \emph{strong} EAC conjecture of Zilber, alluded to above, incorporates a transcendence condition: given any finitely generated subfield $K$ of $\C$, one asks for $(\zbar,\exp_S(\zbar))$ which is generic in $V$ over $K$. It seems likely that one could deduce transcendence results of this type for $V$ with dominant projection to $\C^n$, assuming the appropriate form of Schanuel's conjecture, in the style of \cite{daquino--fornasiero-terzo}. In fact, for $S = \gm$, if one assumes both Schanuel's conjecture and the Zilber-Pink conjecture, then the strong EAC is equivalent to EAC (\cite[Thm.\ 1.5]{Kirby-Zilber-exp).}

\paragraph{Overview of the paper}
Before developing the general theory, we outline the method of proof for Example~\ref{example prop} in Section~\ref{example section}. The multivalued nature of the maps becomes clear as we only have to take square roots and cube roots to describe the corresponding algebraic map $\alpha$.

In Section~\ref{alg map section} we explain how we extract the algebraic map $\alpha$ from the algebraic variety $V$. As we are interested in the behaviour as we approach infinity, we take care in explaining how $\alpha$ can be extended continuously to these points at infinity (in the projective space $\P_n$) where it may fail to be analytic. The content of this section is folklore, but we give a self-contained account. 

We state and prove Theorem~\ref{thm-main-abelian} explaining the solution map $S$ and its properties in Section~\ref{abelian case section}. Theorem~\ref{dominant abelian case} follows, and this also covers Example~\ref{example prop}. In Section~\ref{alg torus section} we indicate how to adapt this work for the algebraic torus setting, to give a new proof of Theorem~\ref{algebraic torus case}. We close with some final remarks.

\section{An example}\label{example section}

Consider the equation
\begin{equation}\label{eqn-wp}
    \wp'(\wp(z)^2) = z,
\end{equation}
from Example~\ref{example prop}, where $\wp$ is a Weierstrass $\wp$-function. Our analysis will work uniformly for any $\wp$-function, but to be definite, let $\wp$ be the Weierstrass $\wp$-function associated to the lattice $\Lambda:= \Z+i\Z$.

We want to determine whether \eqref{eqn-wp} has any solutions in $\C$ and, if so, where those solutions are.

It is well known that 
\[ \wp'(z)^2 = 4\wp(z)^3 - g_2\wp(z) - g_3,  \]
for certain $g_2, g_3 \in \C$, and that the map
\[  z \mapsto [1: \wp(z) : \wp'(z)] \]
gives an embedding of $\C/\Lambda$ into the projective space $\P_2$. The image of $\C/\Lambda$ is an elliptic curve $E\seq \P_2$, and the above map is its exponential map $\exp_E : \C \to E$. 

To exploit the geometry of elliptic curves, we consider $\wp$ and $\wp'$ in \eqref{eqn-wp} as components of the exponential map. However, for simplicity we will write equations and maps in affine coordinates. In particular, if $O := [0:0:1]$ is the point at infinity of $E$ (which is also the identity element of the group structure of $E$), then the affine part of $E$ is $E\setminus \{ O\}\subseteq \C^2$.

Consider the following system of equations:

\begin{equation}\label{eqn-wp-system-affine}
\begin{cases}
x_k = \wp(z_k),\ y_k = \wp'(z_k) & \text{for } k=1,2,\\
y_k^2 = 4x_k^3 - g_2x_k - g_3 & \text{for }k=1,2,\\
z_2 = x_1^2, \\
z_1 = y_2.
\end{cases}\,
\end{equation}
The equations on the second line of \eqref{eqn-wp-system-affine} state that $(x_1,y_1)$ and $(x_2,y_2)$ lie on $E$. 
When combined with the last two equations, they define a subvariety $V$ of $\C^2 \times E^2$. The solutions of \eqref{eqn-wp-system-affine} are the points $(\zbar, \exp_{E^2}(\zbar)) \in V$. One can easily verify that the coordinate $z_1$ of such a point is a solution of \eqref{eqn-wp}, and that all solutions of \eqref{eqn-wp} arise in this way.

We think of $V$ as expressing a point $\wbar \in E^2$ as an algebraic function $\alpha = (\alpha_1,\alpha_2)$ of $\zbar$, that is, we have
\begin{align*}
    \alpha_1(\zbar) = & \left(\sqrt{z_2} , \sqrt{4z_2^{3/2}-g_2z_2^{1/2}-g_3}\right),\\
    \alpha_2(\zbar) = & \left(\beta(z_1)  , z_1 \right),
\end{align*}
where \[ \beta(z_1) := \sqrt[3]{\frac{g_3-z_1^2}{8} + \sqrt{\frac{(g_3-z_1^2)^2}{64}- \frac{g_2^3}{1728}}} + \sqrt[3]{\frac{g_3-z_1^2}{8} - \sqrt{\frac{(g_3-z_1^2)^2}{64}- \frac{g_2^3}{1728}}} \]
is obtained by solving the cubic equation $4v^3 - g_2v - g_3 = z_1^2$ with respect to $v$.

To be more precise, we have to choose single-valued branches of the square and cube roots which we can do by restricting $\alpha$ to a suitable simply connected domain $D \seq \C^2$. Since $\alpha_1$ and $\alpha_2$ depend only on $z_2$ and $z_1$ respectively, we can define them separately. For $k=1,2$, let $N_k\seq \C$ be a closed disc around the origin containing the zeroes of the expressions appearing in the square and cube roots involved in $\alpha_k$.  Also, consider the line (a branch cut) $B:=\R^{<0} \seq \C$. Now set $D_k:=\C \setminus (N_k \cup B)$. Then $D_1$, $D_2$ are simply connected, and we find $4$ branches of $\alpha_1$ and $3$ branches of $\alpha_2$ respectively on $D_1$ and $D_2$. Altogether, we get $12$ branches of $\alpha$ on $D:= D_1\times D_2$, and we pick one of those.

We take a fundamental domain \[ M:= \{ x+iy: -1/2 < x, y \leq 1/2 \} \] for $\exp_E$. Let $\Log_{E^2}: E^2 \to M^2$ be the logarithmic map for this domain. 

Now pick $\zbar\in D$ with $|z_1|, |z_2|$ sufficiently large.  Thus, given some metric inducing the complex topology on $\P_2$, we can say that $\alpha(\zbar)$ is close to the point at infinity $(O,O)$. So $\Log_{E^2} \alpha(\zbar) \approx (0,0) \in \C^2$.

Define a map 
\begin{equation*}
    F: D \to \C^2 : \zbar \mapsto \zbar - \Log_{E^2}\alpha(\zbar).
\end{equation*}
Asymptotically we have $F(\zbar) = \zbar + o(1)$ as $|z_1|, |z_2| \to \infty$. So $F$ is locally invertible and indeed, shrinking $D$ if necessary and staying away from the boundary, we find a connected open set $\tilde{D}$, which is in fact the image of $D$ under $F$, and a map $S: \tilde{D} \to D$ which is the inverse of $F$. Moreover, we also have 
\begin{equation}\label{eq-S-Log}
    S(\zbar) = \zbar + \Log_{E^2} \alpha(\zbar) + o(1) = \zbar + o(1)
\end{equation}
for $|z_1|, |z_2| \to \infty$ with $\zbar \in \tilde{D}$.

Now for a lattice point $\lambar \in \Lambda^2 \cap \tilde{D}$ we have 
\begin{equation}\label{eq-S-F}
    \lambar = F(S(\lambar)) = S(\lambar) -\Log_{E^2} \alpha(S(\lambar)) ,
\end{equation} 
hence $\alpha(S(\lambar)) = \exp_{E^2} (S(\lambar))$. Therefore, the point $S(\lambar)$ is a solution to the equation $\alpha(\zbar) = \exp_{E^2}(\zbar)$, and in fact $\{ S(\lambar): \lambar \in \Lambda^2 \cap \tilde{D} \}$ is the set of all solutions of \eqref{eq-S-F} in $D$.

Finally, one observes that since $F$ is asymptotically the identity, the set $\tilde{D}$ almost contains $D$, in the sense that every point of $D$ sufficiently far from the boundary of $D$ must be in $\tilde{D}$. It follows that $\tilde{D} \cap \Lambda^2$ is not empty, and in fact contains most points of $D \cap \Lambda^2$.
This proves that \eqref{eqn-wp-system-affine} has solutions. In particular, \eqref{eqn-wp} does too: if $S = (S_1,S_2)$ with $S_1, S_2 : \tilde{D} \to \C$, then for each $\lambar \in \Lambda^2$ the element $S_1(\lambar)$ is a solution to \eqref{eqn-wp}.

By repeating the argument for all the possible branches of $S$, and by rotating the branch cuts of $D_1,D_2$, one can verify that all the solutions of \eqref{eqn-wp-system-affine} with $z_1, z_2$ sufficiently large arise in this way.
Since $D$ is a large subset of $\C^2$, these actually give almost all the solutions of \eqref{eqn-wp} such that $z$ and $\wp(z)$ are both sufficiently large. 

Furthermore, we can use the parametrisation of the large solutions of \eqref{eqn-wp-system-affine} by $S$ to understand their geometric distribution. From \eqref{eq-S-F} we have
\begin{align*}
    \lambda_1 &= S_1(\lambar) - \wp^{-1}\left(\sqrt{S_2(\lambar)}\right),\\
    \lambda_2 &= S_2(\lambar) - (\wp')^{-1}(S_1(\lambar)).
\end{align*}
Therefore
\[ S_1(\lambar) = \lambda_1+ \wp^{-1}\left(\sqrt{\lambda_2 + (\wp')^{-1}(S_1(\lambar))}\right). \]
From this and from \eqref{eq-S-Log}, we may conclude that
\begin{align*}
    S_1(\lambar) &= \lambda_1+ \wp^{-1}\left(\sqrt{\lambda_2 + (\wp')^{-1}(\lambda_1 + o(1))}\right) \\
                 &= \lambda_1+ \wp^{-1}\left(\sqrt{\lambda_2 + (\wp')^{-1}(\lambda_1)}\right) + o\left(\wp^{-1}\left(\sqrt{\lambda_2 + (\wp')^{-1}(\lambda_1)}\right)\right)
\end{align*}
as $|\lambda_1|, |\lambda_2| \to \infty$ with $\lambar\in \Lambda^2\cap \tilde{D}$.

Since $\wp$ and $\wp'$ are two-to-one and three-to-one respectively on the fundamental domain $M$, and $z \mapsto z^2$ is two-to-one, one can easily verify that $S_1$ takes 12 distinct values on each $(\lambda_1,\lambda_2)$ according to the choice of the branches.

\section{Algebraic maps}\label{alg map section}

The proof of Theorem~\ref{algebraic torus case} in \cite{brown-masser} and in particular the account of \cite{daquino--fornasiero-terzo} make use of \emph{algebraic functions}. In \cite{daquino--fornasiero-terzo}, those are defined as analytic functions $\alpha : D \to \C$ over some domain $D \subseteq \C^n$ satisfying a non-trivial polynomial equation $P(\zbar,\alpha(\zbar))=0$. In both papers, the authors restrict the choice of the domains $D$ in order to have the appropriate asymptotic behaviour at infinity.

We shall reduce our reliance on detailed asymptotic estimates at infinity in favour of topological and geometric considerations. In Theorems~\ref{dominant abelian case},~\ref{algebraic torus case}, we may assume that $\dim V = n$ after taking some intersections with generic hyperplanes. We then consider the Zariski closure $\Vbar$ of $V$ inside $\P_n \times A$, where $A$ is either the given abelian variety or a suitable completion of $\mathbb{G}_{\mathrm{m}}^n$, and $\C^n$ is embedded into $\P_n$ is the usual way.

In this setting, the projection $\pi : \Vbar \to \P_n$ is surjective, and by dimension considerations, all of its fibres are finite outside of a proper Zariski closed subset of $\P_n$.

We then work with continuous maps $\alpha : D^* \subseteq \P_n \to \Vbar \to A$ such that $(\zbar,\alpha(\zbar)) \in \Vbar$ for all $\zbar \in D^*$, where $D^*$ is some set to be specified later. We can create such a map by composing a continuous section of $\pi : \Vbar \to \P_n$ with the projection $\Vbar \to A$. We wish to understand the behaviour of such an $\alpha$ at infinity, namely at the points of $\P_n \setminus \C^n$.

It is well known that the the analytic local sections $ D \subseteq \P_n \to \Vbar$, where $D$ are suitable open domains, form a sheaf of complex analytic maps. The domains, however, only cover a Zariski open dense subset of $\P_n$, which may well omit all of the points at infinity. We remedy this by taking continuous extensions to some $D^* \supseteq D$ containing points on the boundary of $D$.

We thus obtain maps $\alpha : D^* \to A$ which are continuous, but possibly not analytic, at the points at infinity. The continuity at infinity will encode the asymptotic information needed for the proof of Theorem~\ref{dominant abelian case}.
One could perhaps perform some local resolution of singularities in the style of Bierstone and Milman \cite{bierstone-milman} to make the maps analytic everywhere, but it is not necessary.

The use of continuous extensions of analytic maps is classical, but for the sake of clarity, we state Proposition~\ref{alg map prop} below to pin down which maps we use, and we provide a self contained proof referring to elementary algebraic geometry and algebraic topology. We also make some definite choices of neighbourhoods and sets to ensure we always deal with clearly defined single-valued functions.

Before going further into the technical details, let us work with an elementary example. Identify $\P_1$ with the Riemann sphere $\C \cup \{\infty\}$. Let $\rho : \P_1 \to \P_1$ be the map $z \mapsto z^2 + 1$. The fibres of $\rho$ have cardinality $2$, except over the branching points $1$ and $\infty$, which have fibres $\{0\}$ and $\{\infty\}$ respectively. The restriction of $\rho$ to $\P_1 \setminus \{0,\infty\} = \C \setminus \{0\}$ is a \emph{covering map}: for each $w \in \C \setminus \{1\}$ there is an open neighbourhood $D$ of $w$ in the complex topology such that $\rho^{-1}(D)$ splits into a disjoint union of open sets $D_i$ where $\rho\restrict{D_i}$ is a homeomorphism between $D_i$ and $D$.

If we remove a branch cut, for instance by taking $D = \C \setminus \R_{\geq 0} \subseteq \P_1 \setminus \{0,\infty\}$, we obtain a simply connected domain, hence by standard topological arguments there are two analytic sections $\iota_1,\iota_2 : D \to \P_1$ of $\rho$, and $\rho^{-1}(D) = \iota_1(D) \cup \iota_2(D)$. On the other hand, it is clear that each section can be extended to a continuous section $\iota^* : D^* = D \cup \{0,\infty\} \to \P_1$ by setting $\iota^*(0) = 1$ and $\iota^*(\infty) = \infty$. Such an extension is unique, but it is not analytic at $0$ and $\infty$.

We shall use a higher-dimensional version of the above construction.

\begin{notation}
For $\ell = 0, \dots,n$, let $U_\ell$   be the usual affine chart defined by 
\[U_\ell := \{[z_0:\dots:z_n] \in \P_n : z_\ell = 1\}.\]
We identify $\C^n$ with $U_0$ via the embedding
\begin{equation}\label{Cn into Pn}
    \C^n \ni (z_1,\dots,z_n) \mapsto [1:z_1:\dots:z_n] \in U_0 \subseteq \P_n.
\end{equation}

Now fix a chart $U_\ell$ with $\ell \geq 1$. Given a point $\cbar = [0:c_1:\dots:c_n] \in U_\ell \subseteq \P_n$ (written with $c_\ell =1$), a \emph{polydisc} centred at $\cbar$ in the chart $U_\ell$ of radius $\varepsilon > 0$ takes the form
\begin{equation}\label{eq:polydisc}
 D^*=   \left\{\left[x_0 : x_1 : \cdots : x_n\right] \in \P_n : |x_0| < \varepsilon, |x_i - c_i| < \varepsilon \text{ for }i=1, \ldots, n \text{ and } x_\ell = 1\right\}.
\end{equation}
The intersection of $D^*$ with $\C^n = U_0$ is then
\begin{equation}\label{eq:domain D}
D = \left\{(z_1,\ldots,z_n) \in \C^n : |z_\ell| > \varepsilon^{-1}, \left|\frac{z_i}{z_\ell} - c_i\right| < \varepsilon \text{ for }i=1, \ldots, n \right\}.
\end{equation}
So the variable $z_\ell$ is going to infinity in the annulus given by $|z_\ell| > \varepsilon^{-1}$, and each other coordinate $z_i$ lies in a disc around $c_i z_\ell$ of radius $\varepsilon |z_\ell|$.

We shall work with sectors of the annulus in order to have simply connected domains for our maps. So, for $\theta \in \R$ and $\eta \in (\theta,\theta+2\pi]$ we define
\begin{equation}\label{eq:sector domain}
\Dte = \left\{(z_1,\ldots,z_n) \in D : \theta < \arg (z_\ell) < \eta \text{ for some choice of } \arg(z_\ell)\right\}
\end{equation}
and we extend to the points at infinity with
\begin{equation}\label{eq:sector domain 2}
\Dte^* = \Dte \cup \left(D^* \setminus D\right).
\end{equation}
We will call both $\Dte$ and $\Dte^*$ \emph{sector domains}.
\end{notation}

Note that indeed $\Dte$ and $\Dte^*$ are simply connected, and in fact contractible, as they are homeomorphic to respectively  $\left(\varepsilon^{-1},\infty\right) \times (\theta,\eta) \times \{z \in \C : |z| < \varepsilon\}^{n-1}$ and $\{z \in \C : |z| < \varepsilon, -\eta < \arg(z) < -\theta \text{ or } z = 0\} \times \{z \in \C : |z| < \varepsilon\}^{n-1}$.

We can now state the key result of this section which will be used in the proofs of Theorems~\ref{dominant abelian case} and~\ref{algebraic torus case}. In the following, recall that given a morphism $\rho : X \to Y$ of algebraic varieties, the \emph{degree} of $\rho$, denoted $\deg(\rho)$, is the cardinality of the generic fibre, when it is finite.

\begin{prop}\label{alg map prop}
    Let $A$ be a complete variety and $\Vbar \subseteq \P_n \times A$ be an irreducible variety of dimension $n$ with surjective projection to $\P_n$ of degree $d$. Let $H \leteq \P_n \setminus \C^n$ be the hyperplane at infinity (with $\C^n$ embedded as in \eqref{Cn into Pn}).

    Then there is a Zariski open dense subset $C$ of $H$ with the following property: for all $\ell \geq 1$, $\cbar = [0:c_1:\dots:c_n] \in C \cap U_\ell$, $\theta \in \mathbb{R}$, $\eta\in (\theta,\theta+2\pi]$, and all sufficiently small polydiscs $D^*$ at $\cbar$ in the chart $U_\ell$,
there are  distinct continuous maps $\alpha_1,\dots,\alpha_d : D_{(\theta,\eta)}^* \to A$ such that 
\begin{enumerate}
    \item for all $\zbar \in D_{(\theta,\eta)}^*$ we have $(\zbar,\albar_i(\zbar)) \in \Vbar$;
    \item for all $(\zbar,\wbar) \in \Vbar$ with $\zbar \in D_{(\theta,\eta)}^*$ there is a $k$ such that $\alpha_k(\zbar)=\wbar$;
    \item each restriction $\alpha_i\restrict{D_{(\theta,\eta)}}$ is complex analytic.
    \end{enumerate}
\end{prop}

Note that the algebraic functions of \cite{brown-masser,daquino--fornasiero-terzo} are simply the coordinates of the maps $\alpha_i$ when restricted to maps $\alpha_i : V \to \mathbb{G}_{\rm m}^n$; our sector domains are definite instances of the ``cones'' mentioned in those papers. The additional precision is to avoid potential ambiguities. For instance, in \cite[p1397]{daquino--fornasiero-terzo} the authors claim that every algebraic function is asymptotically homogeneous, which however may be false if the cones include real lines pointing outside of $C$.

\begin{remark}\label{sheaf remark}
    While Proposition~\ref{alg map prop} is fairly detailed in the use of specific polydiscs and sector domains, one may also read it as the construction of a particular sheaf.
    
    Recall that the local analytic sections of the projection $\Vbar \to \P_n$ with open domains form a sheaf of functions over $\P_n$. One can think of the sheaf as the collection of all continuations of any one local section, and so as a multivalued analytic function. The collection of the maps $\alpha_i$ from Proposition~\ref{alg map prop}, as $\theta, \eta$ vary, is essentially the composition of this multivalued function with the projection $\Vbar \to A$.

    First consider the map $\alpha_1$ on the domain $D_{(0,2\pi)}$, which we now write as $\alpha^1_{(0,2\pi)}$ and all the $d$ maps $\alpha_k$ on $D_{(\pi,3\pi)}$. For each $k$, the set $\left\{\zbar \in D_{(\pi,2\pi)} : \alpha^1_{(0,2\pi)}(\zbar) = \alpha_k(\zbar) \right\}$ is clopen by the uniqueness of analytic continuation, so  exactly one of these $\alpha_k$ must agree with $\alpha^1_{(0,2\pi)}$ on $D_{(\pi,2\pi)}$, and we write it as $\alpha^1_{(\pi,3\pi)}$. 

    Similarly, analytic continuation determines a unique branch of $\alpha$ on $D_{(j\pi,(j+2)\pi)}$ for each $j \in \Z$, which we write as $\alpha^1_{(j\pi,(j+2)\pi)}$. Continuity of each branch of $\alpha$ on $\Dte^*$ ensures that $\alpha^1$ extends uniquely to $D^*\setminus D$, and is single valued there.
    Since there are only $d$ branches of $\alpha$ at each point, for some integer $e$ with $0< e< d$ we must have 
    $\alpha^1_{(2e\pi ,(2e+2)\pi)} = \alpha^1_{(0, 2\pi)}$. So $\alpha^1$ is a multivalued map $D^* \to A$ which is $e$-valued on $D$ and single-valued on $D^* \setminus D$. This ramification is exactly analogous to the function $z \mapsto z^{1/e}$ about the point at infinity in $\P_1$.
 
    If $e<d$ then we can continue with the other branches of $\alpha_k$ to get connected multivalued maps $\alpha^1,\ldots,\alpha^{d'}:D^* \to A$, corresponding to the $d'$ different values of $\alpha(\cbar)$. We can then consider these maps together as a single $d$-valued map $\alpha: D^* \to A$.

    We can also patch together the maps defined on sector domains around different points $\cbar \in C$. Overall, the sector domain cover an open subset $\Omega^*$ of $\P_n$ with $\Omega^* \setminus \C^n = C$, and the maps $\alpha : \Dte \to A$ generate a sheaf $\alphabar$ representing a $d$-valued analytic map on $\Omega = \Omega^* \cap \C^n$. The additional information in Proposition~\ref{alg map prop} encodes how each branch extends continuously to $C$.
    
   Ahlfors explains the construction of algebraic functions in detail, although only in a single variable, in \cite[pp284--308]{ahlfors}. 
    
    Later we will do something similar with other analytic maps.

\end{remark}

The rest of this section offers a fairly detailed proof of Proposition~\ref{alg map prop}, but we stress that the construction is folklore and that the techniques used here will not be relevant for the rest of the paper, and so the reader may well skip to the next section.

\subsection{Covering maps}
First, we recall how generically finite maps between irreducible varieties restrict to topological covering maps. These are classical facts, and we refer the reader to \cite[\S IV.2]{hartshorne} and \cite[\S II.6]{Shafarevich} for more details. In the following, let $X$, $Y$ be complete irreducible algebraic varieties of the same dimension, and $\rho : X \to Y$ be a \emph{surjective} morphism.

\begin{proposition}\label{prop:theta-finite}
    Let $Y_f = \{y \in Y : |\rho^{-1}(y)| < \infty\}$. Then $Y_f$ is Zariski open dense in $Y$ and $Y \setminus Y_f$ has codimension at least $2$ in $Y$. 
\end{proposition}
\begin{proof}
    Since $X$, $Y$ are complete, $f$ is a closed morphism. Therefore, the conclusion is an immediate consequence of Chevalley's semi-continuity theorem \cite[Cor.\ 13.1.5]{EGA-IV-3} (or \cite[Thm.~1.25]{Shafarevich}, after correcting the statement by adding the word `closed'): the set of points $y \in Y$ such that $\rho^{-1}(y)$ has dimension at least $1$ is a proper and Zariski-closed subset of $Y$. By the same theorem, since $\rho^{-1}(Y \setminus Y_f)$ has dimension at most $\dim(Y) - 1$, and all the fibres over $Y \setminus Y_f$ have dimension at least $1$, we must have $\dim(Y \setminus Y_f) < \dim(Y) - 1$.
\end{proof}

\begin{proposition}\label{prop:theta-covering}
    Let $Y_c = \{y \in Y_f : y \text{ is non-singular}, |\rho^{-1}(y)|=\deg(\rho)\}$. Then $Y_c$ is Zariski open dense in $Y_f$ and the restriction $\rho\restrict{\rho^{-1}(Y_c)}$ is a covering map with respect to the complex topology.
\end{proposition}
\begin{proof}
    Since $X$, $Y$ are complete, $\rho$ is proper, and so is $\rho\restrict{\rho^{-1}(Y_f)}$ as $Y_f$ is open. The latter map is quasi-finite (that is, it has finite fibres), hence it is finite in the sense of algebraic geometry \cite[Thm.\ 8.11.1]{EGA-IV-3}. We also know that the singular points of $Y$ form a proper Zariski closed subset of $Y$, and they are normal. By \cite[Theorem 2.29]{Shafarevich}, the normal points of $Y_f$ such that $|\rho^{-1}(y)| < \deg(\rho)$ form a proper Zariski closed subset of $Y_f$, and in particular do so the non-singular ones. One can then verify that $\rho\restrict{\rho^{-1}(Y_c)}$ is a covering map via the implicit function theorem (see \cite[p.\ 143]{Shafarevich}).
\end{proof}

Note in particular that $\rho\restrict{\rho^{-1}(Y_c)}$ is an open map in both complex and Zariski topologies.

\begin{corollary}\label{cor:sections}
    Let $D \subseteq Y_c$ be simply connected. Let $d = \deg(\rho)$. Then there are exactly $d$ complex analytic sections $\iota_1,\dots,\iota_d : D \to X$ of $\rho$, and $\rho^{-1}(D) = \bigcup_{i=1}^d \iota_i(D)$.
\end{corollary}
\begin{proof}
    This is a standard algebraic topology result. Fixed $y \in D$ and $x \in \rho^{-1}(y)$, there is an analytic homeomorphism from a neighbourhood of $y$ to a neighbourhood of $x$. Since $D$ is simply connected, such a homeomorphism has a continuation to all of $D$. Any two sections $\iota$, $\iota'$ such that $\iota(y)=\iota'(y)=x$ have the property that $\{ y' \in D : \iota(y') = \iota'(y') \}$ is both open (by analytic continuation) and closed (by continuity of the map $\iota \times \iota' : D \times D \to X \times X$ and the fact that the diagonal of $X \times X$ is closed, since $X$ is Hausdorff), thus they coincide. In turn, there is exactly one section for every point in the fibre $\rho^{-1}(y)$. By repeating the argument on all $y \in D$, one sees that images of these sections cover all of $\rho^{-1}(D)$.
\end{proof}

\subsection{Extending sections of covering maps}
We now wish to extend continuously the sections $\iota_i$ from a simply connected $\Delta \subseteq Y_c$ to some larger domain $\Delta^*$ which may fall outside of $Y_c$. In Proposition~\ref{alg map prop}, these will be the points at infinity in $\Dte^* \setminus \Dte$.

Whether this can be done depends on the topological properties of $\Delta^*$, rather than the algebraic properties, so we will work in an abstract topological setting.

\begin{lemma}\label{lem:extend-local-inverse}
    Let $X$, $Y$ be compact Hausdorff topological spaces, $\rho : X \to Y$ be a continuous function, and $\iota : \Delta \to X$ be a continuous section of $\rho$ on some $\Delta \subseteq Y$. Let $\Delta^* \subseteq Y$ be such that $\Delta \subseteq \Delta^* \subseteq \overline{\Delta}$.
    
    Suppose that for all $y \in \Delta^* \setminus \Delta$:
    \begin{itemize}
        \item $\rho^{-1}(y)$ is finite;
        \item there are arbitrarily small neighbourhoods $N$ of $y$ such that $N \cap \Delta$ is connected.
    \end{itemize}
    Then $\iota$ extends (uniquely) to a continuous section $\iota^* : \Delta^* \to X$ of $\rho$.
\end{lemma}
\begin{proof}
    Fix some $y \in \Delta^* \setminus \Delta$. Since $X$ is Hausdorff, we may find an open neighbourhood $B = B_y$ of $\rho^{-1}(y)$ where each connected component contains exactly one point of $\rho^{-1}(y)$.

    \begin{claim*}
        For every neighbourhood $N'$ of $y$, there is a neighbourhood $N \subseteq N'$ of $y$ such that $\rho^{-1}(N) \subseteq B$ and $N \cap \Delta$ is connected.
    \end{claim*}
    \begin{proof}
        Since $Y$ is compact Hausdorff, $N'$ contains closed neighbourhoods $N'' \subseteq N'$ of $y$. Suppose by contradiction that $\rho^{-1}(N'') \nsubseteq B$ for all such $N''$'s. By compactness of $X$, the intersection of the (closed) sets $\rho^{-1}(N'') \setminus B$ contains some $x \notin B \supseteq \rho^{-1}(y)$; since $Y$ is Hausdorff, we may pick a closed neighbourhood of $y$ not containing $\rho(x)$, a contradiction. By assumption, there is a neighbourhood $N \subseteq N''$ of $y$ such that $N \cap \Delta$ is connected, and of course $\rho^{-1}(N) \subseteq \rho^{-1}(N'') \subseteq B$.
        \let\oldqedsymbol=\qedsymbol
        \renewcommand{\qedsymbol}{\oldqedsymbol\textsubscript{Claim}}
    \end{proof}

    Let $N$ be any neighbourhood of $y$ given by the claim. Since $\Delta^* \subseteq \overline{\Delta}$ and $N$ is a neighbourhood of $y$, $N \cap \Delta$ is non-empty. Thus $\iota(N \cap \Delta)$ is contained in exactly one connected component $B_0$ of $B$. By construction of $B$ we have $\overline{B_0} \cap (B \setminus B_0) = \emptyset$, thus $\overline{\iota(N \cap \Delta)} \cap \rho^{-1}(y) \subseteq \overline{B_0} \cap \rho^{-1}(y) = \{x_y\}$ for some $x_y \in \rho^{-1}(y)$.

    If $\iota^*$ is a continuous extension of $\iota$, since $y \in N \cap \overline{\Delta} \subseteq \overline{N \cap \Delta}$, we must have $\iota^*(y) \in \overline{\iota(N \cap \Delta)}$. In turn, if $\iota^*$ is a also section of $\rho$ we must have $\iota^*(y) \in \overline{\iota(N \cap \Delta)} \cap \rho^{-1}(y)$, hence $\iota^*(y) = x_y$.
    
    Therefore, if continuous sections extending $\iota$ exist, they are unique. As for their existence, define the section $\iota^* \supseteq \iota$ by letting $\iota^*(y) = x_y$ for $y \in \Delta^* \setminus \Delta$. It remains to check that $\iota^*$ is continuous.

    Let $N'$ be an open neighbourhood of a $y \in \Delta^* \setminus \Delta$. By the claim and the above argument, there is a neighbourhood $N \subseteq N'$ such that $\iota^*(y) \in \overline{\iota(N \cap \Delta)} \subseteq \overline{\iota(N' \cap \Delta)}$. By repeating this for every $y' \in N' \cap (\Delta^* \setminus \Delta)$, we get $\iota^*(N' \cap \Delta^*) \subseteq \overline{\iota(N' \cap \Delta)}$.

    Now fix $y \in \Delta^* \setminus \Delta$ and let $B$ be a neighbourhood of $\iota^*(y)$. Since $X$ is compact Hausdorff, there is a closed neighbourhood $B' \subseteq B$ of $\iota^*(x)$. By continuity of $\iota$, we can find a neighbourhood $N'$ of $y$ such that $\iota(N' \cap \Delta) \subseteq B'$. In turn, $\iota^*(N') \subseteq \overline{\iota(N' \cap \Delta)} \subseteq B' \subseteq B$. By definition, this means that $\iota^*$ is continuous at $y$.
\end{proof}

The above can now be applied in the setting of maps between algebraic varieties. The following statement establishes a natural condition on pairs of domains $\Delta \subseteq \Delta^*$ guaranteeing the existence (and uniqueness) of such extensions.

\begin{proposition}\label{prop:extended-sections}
    Let $X$ be a complete irreducible algebraic variety and $\rho : X \to \P_n$ be a morphism of degree $d$. Let $\Delta \subseteq \Delta^* \subseteq \P_n$ be sets with $\Delta$ simply connected, $\Delta^* \subseteq \overline{\Delta}$ (where $\overline{\Delta}$ is the topological closure of $\Delta$), such that for all $y \in \Delta^*$:
    \begin{itemize}
        \item $\rho^{-1}(y)$ is finite, and of cardinality $d$ when $y \in \Delta$;
        \item there are arbitrarily small neighbourhoods $N$ of $y$ such that $N \cap \Delta$ is connected.
    \end{itemize}
    Then there are exactly $d$ continuous sections $\iota_1,\dots,\iota_d : \Delta^* \to X$ of $\rho$; moreover, they are complex analytic on $\Delta$ and $\rho^{-1}(\Delta^*) = \bigcup_{i=1}^d \iota_i(\Delta^*)$.
\end{proposition}
\begin{proof}
    Since $\P_n$ is non-singular, on setting $Y = \P_n$ we have that $\Delta \subseteq Y_c$. By Corollary~\ref{cor:sections}, there are sections $\iota_1',\dots,\iota_d' : \Delta \to X$ satisfying the conclusion with $\Delta$ in place of $\Delta^*$.
    
    Since $X$ and $\P_n$ are complete complex varieties, they are compact Hausdorff spaces, so by Lemma~\ref{lem:extend-local-inverse}, such sections can be extended uniquely to continuous sections $\Delta^* \to X$.

    It remains to verify that for every $y \in \Delta^*$, $\rho^{-1}(y) = \{\iota_1(y), \dots, \iota_d(y)\}$. Suppose by contradiction that there is $x$ such that $\rho(x) = y \in \Delta^*$, but $\rho(x) \neq \iota_i(y)$ for all $i = 1, \dots, d$. By assumption, $\Delta^* \subseteq Y_f$, and $Y_f$ is open by Proposition~\ref{prop:theta-finite}. By Remmert's open mapping theorem \cite[\S V.6, Theorem 2]{Loj-analytic-geom}, $\rho \restrict {\rho^{-1}(Y_f)}$ is an open map. If $B \subseteq \rho^{-1}(Y_f)$ is an open neighbourhood of $x$ not containing any $\iota_i(y)$, then $\rho(B)$ is an open neighbourhood of $y$ disjoint from $\Delta$, a contradiction since $\Delta^* \subseteq \overline{\Delta}$.
\end{proof}

\subsection{The algebraic maps}

We return to the setting of Proposition~\ref{alg map prop}, with the sets $D^*$, $D$, $\Dte$ and $\Dte^*$ as given in \eqref{eq:polydisc}, \eqref{eq:domain D}, \eqref{eq:sector domain}, and \eqref{eq:sector domain 2}.

\begin{proof}[Proof of Proposition~\ref{alg map prop}]
    Let $\pi$ be the projection $\Vbar \to \P_n$. Let $Y_c \subseteq \P_n$ be the set of the points $y$ such that $\pi^{-1}(y)$ is finite of cardinality $d = \deg(\pi)$ (this coincides with the set $Y_c$ of Proposition~\ref{prop:theta-covering} on letting $Y = \P_n$ and $\rho = \pi$). We let
    \be \label{eq:C}
    C := H \setminus \left(\{ y \in H : \pi^{-1}(y) \text{ is infinite}\} \cup \overline{\C^n \setminus Y_c}^{\mathrm{Zar}}\right)
    \ee
    where $\overline{(\cdot)}^{\mathrm{Zar}}$ denotes the Zariski closure in $\P_n$.
    
    By construction, $Y_c \cup C$ is a Zariski open subset of $\P_n$. Moreover, $C$ is non-empty: $\C^n \setminus Y_c$ has dimension at most $n-1$ by Proposition~\ref{prop:theta-covering}, thus $\overline{\C^n \setminus Y_c}^{\mathrm{Zar}} \cap H \subseteq \overline{\C^n \setminus Y_c}^{\mathrm{Zar}} \setminus (\C^n \setminus Y_c)$ has dimension at most $n-2$, and likewise the set of points $y$ such that $\theta^{-1}(y)$ is infinite has dimension at most $n-2$ by Proposition~\ref{prop:theta-finite}.
    
    Let $\cbar \in C$, $U_\ell$ be a chart such that $\cbar \in U_\ell$, and let $D^*$ be a polydisc centred at $\cbar$ in $U_\ell$. When $D^*$ is sufficiently small, $D^* \subseteq Y_c \cup C$, since $Y_c \cup C$ is open. We shall assume this to be the case.
    
    We now wish to apply Corollary~\ref{cor:sections}. Recall that the sector domains $\Dte$, $\Dte^*$ are simply connected. Moreover, for any $\zbar \in D^* \setminus D$, and any polydisc $D'$ centred at $\zbar$ in $U_\ell$ and such that $D' \subseteq D^*$, we clearly have $\Dte' = D' \cap \Dte$, and that is connected. Since $D_{(\theta,\eta)}$ is open and locally connected, this shows that any $\zbar\in D_{(\theta,\eta)}^*$ has arbitrarily small neighbourhoods $N$ such that $N \cap D_{(\theta,\eta)}$ is connected, as desired.

    Thus, we can apply Proposition~\ref{prop:extended-sections} and obtain sections $\iota_1, \dots, \iota_d$ of $\pi$ with domain $\Dte^*$. Their composition with the projection from $\Vbar$ to $A$ are the desired maps $\alpha_1,\dots,\alpha_d$: since $\iota_i(\zbar) = (\zbar, \alpha_i(\zbar))$, conclusion (1) follows at once from $\iota_i$ being a section of $\pi$; (2) holds by $\pi^{-1}(D) = \bigcup_{i=1}^d \iota_i(D)$; (3) follows from Corollary~\ref{cor:sections} after noticing that $\Dte \subseteq Y_c$.
\end{proof}

For comparison with Section~\ref{example section}, note that the sector domain $D_{(\theta,\theta+2\pi)}^*$ is effectively the polydisc $D^*$ with a branch cut in the variable $z_\ell$ removed. A single branch cut is sufficient: since each coordinate $z_i$ is close to a fixed multiple of $z_\ell$, the branch cut in $z_\ell$ guarantees that the other coordinates also cannot make a loop around $\cbar$. One can easily verify that the second branch cut in Section~\ref{example section} becomes redundant if we add a restriction of the form $|z_1 - z_2| < \varepsilon$, as we do here with the polydiscs.

\section{The abelian case}\label{abelian case section}

Recall that a complex abelian variety is an irreducible projective complex algebraic variety $A$ with a commutative algebraic group structure which makes it also a complex Lie group. Let $n \leteq \dim A$. The exponential map $\exp_A : \C^n \to A$ is a surjective complex analytic homomorphism. Its kernel $\Lambda$, the \emph{period lattice}, is isomorphic (as a topological group) to $\Z^{2n}$, with the discrete topology. We write $\Log_A$ for the local inverse of $\exp_A$.

In this section we prove Theorem \ref{dominant abelian case} stating that a variety $V\seq \C^n \times A$ with a dominant projection to $\C^n$ contains an exponential point. In fact, we prove a stronger result, not only showing the existence of exponential points on $V$ but also locally describing the set of almost all sufficiently large such points.

\begin{theorem}\label{thm-main-abelian} 

Let $A$ be a complex abelian variety of dimension $n$. Let $V\seq \C^n \times A$ be an irreducible subvariety of dimension $n$ with dominant projection $\pi: V \to \C^n$. Let $d:= \deg \pi$. We embed $\C^n$ in projective space $\P_n$ in the usual way \eqref{Cn into Pn}, identifying it with the chart $U_0$.

Then there is a subset $\Omega^* \subs \P_n$, which is open in the complex topology, such that $C\leteq \Omega^* \setminus \C^n$ is Zariski open dense in $\P_n\setminus \C^n$, and there is a sheaf $\mathbf{S}$ of analytic maps on $\Omega := \Omega^* \cap \C^n$ taking values in $\C^n$ with the following properties:

\begin{enumerate}
\item The image $\mathbf{S}(\Omega)$ contains $\Omega$ except possibly for a bounded strip along the boundary $\partial \Omega$.

\item For $\lambar \in \Omega \cap \Lambda$, each
value of $\mathbf{S}(\lambar)$ satisfies $(\mathbf{S}(\lambar),\exp_A(\mathbf{S}(\lambar))) \in V$. Furthermore, these are the only exponential points $(\zbar,\exp_A(\zbar))$ of $V$ with $\zbar$ in $\Omega$ (except possibly near the boundary).

\item These exponential points are locally in $d$-to-$1$ correspondence with the points of $\Lambda \cap \Omega$: $\mathbf{S}$ has $d$ branches, possibly up to translation of the argument by elements of $\Lambda$.

\item The solutions $\mathbf{S}(\lambar)$ are asymptotically translates of the lattice: for each $\cbar \in C$ and branch $S$ of $\Sbar$, there is a $\gambar \in \C^n$ such that $S(\zbar) = \zbar + \gambar + o(1)$ for $\zbar \to \cbar$.

\item In particular, the set $\Sbar(\Omega \cap \Lambda)$ is Zariski dense in $\C^n$, and the set of exponential points $\{ (\zbar,\exp_A (\zbar))\in V \}$ is Zariski dense in $V$.

\end{enumerate}

\end{theorem}
Further properties of the individual maps $S:\Dte \to \P_n$ making up the sheaf, and their extensions to $\Dte^*$, are given in Proposition~\ref{prop:maps S}.

Theorem \ref{dominant abelian case} can easily  be deduced from Theorem~\ref{thm-main-abelian}.

\begin{proof}[Proof of Theorem \ref{dominant abelian case}]
Let $V\seq \C^n\times A$ be an algebraic subvariety with dominant projection to $\C^n$. If $\dim(V) = n$, we are done by Theorem~\ref{thm-main-abelian}.

In general, let $W$ be a proper subvariety of $V$. We can choose a subvariety $H \subseteq A$ of codimension $\dim(V) - n$, such that $V' \leteq V \cap (\C^n \times H)$ is irreducible and not contained in $W$. (For example, $H$ can be taken to be an intersection of sufficiently generic hyperplanes.) In particular, $V' \cap W$ is not Zariski dense in $V'$. Then $\dim V' = n$, and $V'$ has dominant projection to $\C^n$,  so by Theorem~\ref{thm-main-abelian}, $V'$ contains an exponential point outside of $W$. Therefore, the exponential points are Zariski dense in $V$, as desired.
\end{proof}

The proof of Theorem~\ref{thm-main-abelian} will run through the rest of this section. We give a brief summary of the key steps first.

\paragraph{Proof outline.}

The proof proceeds in five steps.
\begin{enumerate}
    \item We use Proposition~\ref{alg map prop} to find the set $C$ and extract a algebraic maps $\alpha: \Dte^* \to A$ from $V$, where $\Dte^*$ ranges over sector domains around points $\cbar \in C$. For simplicity, in this summary we will work with the restrictions of $\alpha$ to $\Dte$, the affine part of $\Dte^*$. 
    
    \item We show that a bounded holomorphic branch of $\Log_A(\alpha(\zbar))$ can be defined on $\Dte$, which we denote by $G : \Dte \to \C^n$. Then we consider the map $F(\zbar):= \zbar - G(\zbar) : \Dte \to \C^n$. A point $\zbar \in \Dte$ satisfies $\exp_A(\zbar) = \alpha(\zbar)$ if and only if $F(\zbar) \in \Lambda$. 
    
    \item We prove, possibly after shrinking $\Dte$, that $F$ is injective on $\Dte$ and its Jacobian matrix of first partial derivatives is non-singular. This implies that $F$ has a holomorphic inverse $S$.
    
    \item We show that each sector domain $D_{(t,s)}$ is covered by the images $F(\Dte)$ as $(\theta,\eta)$ varies. In particular, the image of $F$ contains all the lattice points in a neighbourhood of $\cbar$, and so the solutions we want are the images under $S$ of the lattice points. We also describe the asymptotic behaviour of $S$.
    
    \item Finally, we explain how the local maps $S: \Dte \to \C^n$ are patched together and complete the proof of the theorem.
\end{enumerate}

We remark that Steps~1-4 are already sufficient to prove Theorem~\ref{dominant abelian case}, namely the existence of the exponential points. Step~5 yields the additional distribution of the exponential points towards infinity as described in Theorem~\ref{thm-main-abelian}.

\paragraph{Step 1: The algebraic maps.}\label{subsec:abelian-proof}

We begin by applying Proposition~\ref{alg map prop} and extracting algebraic maps from $\overline{V}$, the projective closure of $V$ in $\P_n \times A$. Let us fix the following data:
\begin{itemize}
    \item $\cbar$ an arbitrary point of $C$, where $C$ is as in Proposition~\ref{alg map prop};
    \item $1 \leq \ell \leq n$ such that $c_\ell \neq 0$;
    \item $D^*$ a small polydisc at $\cbar$ in the chart $U_\ell$ as defined in \eqref{eq:polydisc}.
\end{itemize}

Then for each $\theta \in \R$ and each $\eta \in (\theta,\theta+2\pi]$ we have the sector domains $\Dte \subs \C^n$ and their extensions $\Dte^*$ to $\P_n$ as given in \eqref{eq:sector domain} and \eqref{eq:sector domain 2}.
By Proposition~\ref{alg map prop}, there are (unique) algebraic maps
\begin{equation}
    \alpha_1, \dots, \alpha_d : \Dte^* \to A
\end{equation}
whose graphs cover the points of $\Vbar$ over $\Dte^*$.

Through the proof, we may shrink the $\varepsilon$ used to define $D^*$ in \eqref{eq:polydisc} to ensure certain properties hold for certain maps. When a statement begins with `For small enough $D^*$', the following statements will implicitly assume that $\varepsilon$ is sufficiently small to make that statement true.

For the sake of readability we will drop the subscript and just write $\alpha : \Dte^* \to A$ for the algebraic map.

\paragraph{Step 2: Mapping the solutions to the lattice.}
The goal of this step is to define a map $F$ which maps the solutions of the equation $\exp_A(\zbar) = \alpha(\zbar)$ to lattice points.

Since $\Dte^*$ is simply connected, we can choose a continuous branch on $\Dte^*$ of the (multivalued) composite $\Log_A \circ \alpha$. We pick one such branch and call it $G$. Then we have: 
\begin{gather}
      G : \Dte^* \to \C^n \text{ continuous on } \Dte^* \text{ and holomorphic on } \Dte,\nonumber \\ \text{ such that } \exp_A(G(\zbar)) = \alpha(\zbar) \text{ for all } \zbar \in \Dte^*.  \label{eq:G-alpha}
\end{gather}

As with $\alpha$, we can patch together the maps $G$ as $(\theta,\eta)$ varies, and this patching is uniquely determined by analytic continuation on $\Dte$ and then by continuity on $\Dte^*$. The union of their graphs yields a multivalued map $\mathbf{G}:D^* \to \C^n$, of which the maps $G$ are single-valued branches; the restrictions of the maps to $D$ and their continuations yield a sheaf of analytic maps as in Remark~\ref{sheaf remark}. 

Let $\mu$ be some fixed positive real number.

\begin{proposition}\label{prop:bounded}
  For small enough $D^*$, the image $G(D^*)$ is bounded in $\C^n$.

  Moreover, for small enough $D^*$, every branch $G$ of $\mathbf{G}$ on every sector domain has image contained in an open ball centred at $G(\cbar)$ of radius at most $\mu$.
\end{proposition}
\begin{proof}
  By continuity, $G(\zbar) \to G(\cbar)$ as $\zbar \to \cbar$, so for each $(\theta,\eta)$, by shrinking $D^*$ we may assume that the image $G(\Dte^*)$ falls into an open ball around $G(\cbar)$, with radius at most $\mu$. We can shrink $D^*$ sufficiently so that this bound holds simultaneously for $G$ on all $d$ branches of $\alpha$ and all sector domains $\Dte^*$.
\end{proof}

We now use $G$ to define a new map $F$ which will take the solutions to lattice points.
Define $F : \Dte^* \to \P_n$ by
\be
F(\zbar) \leteq 
\begin{cases}
\zbar - G(\zbar),& \text{ when } \zbar \in \Dte,\\
\zbar, &\text{ when } \zbar \in \Dte^*\setminus \Dte.
\end{cases}
\label{eq for F}
\ee

\begin{proposition}\label{prop F cont infty}
The map $F$ has the following properties.

\begin{itemize}
    \item  $F$ is continuous on $\Dte^*$ and holomorphic on $\Dte$;
    \item $F(\Dte) \subs \C^n$;
    \item A point $\zbar \in \Dte$ is a solution to the equation $\exp_A(\zbar) = \alpha(\zbar)$ if and only if $F(\zbar) \in \Lambda$.
\end{itemize}
\end{proposition}
\begin{proof}
Since $G$ is holomorphic on $\Dte$, it follows that $F$ is also holomorphic there. 
To prove continuity of $F$ at a point $\abar\in \Dte^*\setminus \Dte$, it suffices to observe that $G$ is bounded in $\C^n$, that is, in the chart $U_0$, and so in the natural metric of the chart $U_\ell$ around $\abar$, we have $F(\zbar) - \zbar \to 0$ as $\zbar\to\abar$.

   The second statement is evident and the third statement follows immediately from \eqref{eq:G-alpha} and \eqref{eq for F}.
\end{proof}

Just as for $\mathbf{G}$, the union of the maps $F$ yield a multivalued function $\mathbf{F} : D^* \to \P_n$, which we may suggestively write as $\mathbf{F}(\zbar) = \zbar - \mathbf{G}(\zbar)$.

\paragraph{Step 3: Local injectivity of \texorpdfstring{$\mathbf{F}$}{F}.}

In this step we show that $F$ is injective when $\Dte^*$ is small enough, both in terms of shrinking the polydisc $D^*$ and of moving $\eta$ closer to $\theta$. Hence, as a multivalued function, $\mathbf{F}$ is locally invertible on $D$. 

First, we recall Cauchy's estimate from the theory of complex functions.
See, for example, \cite[Chapter 1, \S 2.6, Theorem 4]{shabat}. 

\begin{fact}[Cauchy Estimate]\label{lem:bounded-deriv}
    Let $f : \Omega \to \C$ be a holomorphic function on an open domain $\Omega\seq \C^n$ containing a closed polydisc $T$ of radius $r$ centred at a point $\wbar \in \Omega$. Then for any $k$ we have
    \[ \left| \frac{\partial f}{\partial z_k}(\wbar) \right| \leq \frac{\max_{\zbar \in T}|f(\zbar)|}{r}. \]
\end{fact}

\begin{proposition}\label{prop:Jac < epsilon}
  For all $\nu > 0$, there is $D^*$ small enough such that for all intervals $(\theta,\eta)$, for all $\zbar \in \Dte$, the norm of the Jacobian matrix of the first partial derivatives $\dd G(\zbar)$ is less than $\nu$.

  In particular, for small enough $D^*$, for all branches $F$, $G$ of respectively $\mathbf{F}$, $\mathbf{G}$ on every sector domain $\Dte^*$, for all $\zbar \in \Dte$ we have:
  \begin{itemize}
    \item $\| \dd G(\zbar) \|< 1/2$,
    \item $\det (\dd F(\zbar)) \neq 0$.
  \end{itemize}
\end{proposition}

To be more precise, here we use the $\ell^{\infty}$-norm on $\C^{n^2}$ as the matrix norm, and denote it by $\| \cdot \|$.

\begin{proof}
  By continuity, we may shrink $D^*$ so that for all $\zbar \in D$ and for all branches $G$ of $\mathbf{G}$ at $\zbar$, we have $|G(\zbar) - G(\cbar)| < \nu$. Here $\mid \cdot \mid$ denotes the $\ell^2$-norm on $\C^{n}$. 

  Furthermore, we may shrink $D^*$ further and assume that for every $\zbar \in D$, every polydisc $T$ of radius 1 around $\zbar$, and every $\zbar' \in T$, we have $|G(\zbar') - G(\cbar)| < \nu$.
  Now choose $\zbar \in D$ and a branch of $\mathbf{G}$ defined on $D_{(\arg(z_\ell)-\pi,\arg(z_\ell)+\pi)}$.
  Then the polydisc $T$ of radius $1$ around $\zbar$ is entirely contained in $D_{(\arg(z_\ell)-\pi,\arg(z_\ell)+\pi)}$.
  
  Then we apply the Cauchy Estimate (Fact \ref{lem:bounded-deriv}) to the coordinate functions of $G(\zbar) - G(\cbar)$ and deduce that their partial derivatives at $\zbar$, which are equal to those of $G(\zbar)$, are bounded by $\nu$. 

  Since $\dd F(\zbar) = I - \dd G(\zbar)$, when $\| \dd G(\zbar) \|$ is sufficiently small, $\dd F(\zbar)$ is close to the identity matrix, hence it is non-singular.
\end{proof}

 We can summarise the above statements with $\|\dd \mathbf{G}(\zbar)\| < \frac{1}{2}$, $\det(\mathbf{F}(\zbar)) \neq 0$ for all $\zbar \in D$, where the inequalities implicitly apply to all values of $\mathbf{G}$ and $\mathbf{F}$.

We now show that $F$ is injective when its domain $\Dte^*$ is sufficiently small, in the sense that $\eta$ is sufficiently close to $\theta$. Recall that by Proposition~\ref{prop:bounded} the set $G(\Dte^*)$ is bounded. Indeed, given any $\xbar, \ybar \in \Dte^*$, and any branch $G$ of $\mathbf{G}$ with that domain, we have $|G(\xbar) - G(\ybar)| < 2\mu$.

\begin{prop}\label{prop:injective}
For small enough $D^*$, there is a small $\delta>0$ such that for all $\theta \in \R$, the map $F$ is injective on $D^*_{(\theta,\theta+2\pi-\delta)}$.
\end{prop}

\begin{proof}
Write $\eta$ for $\theta+2\pi-\delta$, with $\delta$ to be determined later. Since $F$ is the identity on $\Dte^*\setminus \Dte$, it is injective there. Also, $F$ maps $\Dte$ to $\C^n$ which is disjoint from $\Dte^*\setminus \Dte$, so it suffices to show that $F$ is injective on $\Dte$. Suppose $\xbar, \ybar \in \Dte$ are such that $F(\xbar) = F(\ybar)$. Recall that $G(\zbar) = \zbar-F(\zbar)$. So we have
\[ |\xbar - \ybar| = |G(\xbar) - G(\ybar)| < 2\mu.\]
If the line segment $[\xbar,\ybar]$ is entirely contained in $\Dte$ then, by the mean value inequality, 
 \[ |\xbar - \ybar| = |G(\xbar) - G(\ybar)| \leq  \max_{\zbar\in [\xbar,\ybar]}\|\dd G(\zbar)\| \cdot |\xbar-\ybar|. \]
Then by Proposition~\ref{prop:Jac < epsilon} we have $\max_{\zbar\in [\xbar,\ybar]}\|\dd G(\zbar)\|<1/2$, so $\xbar = \ybar$.

Now assume that $[\xbar,\ybar] \nsubseteq \Dte$. First, suppose that the segment crosses the region $D \setminus \Dte$. In particular, it will contain two points $\zbar'$, $\zbar''$ with respectively $\arg(z_\ell') = \theta$ and $\arg(z_\ell'') = \eta$, and we must have $\delta < \pi$. Therefore, $|\zbar' - \zbar''| \geq |z_\ell' - z_\ell''| \geq \varepsilon^{-1} \cdot 2\sin(\frac{\delta}{2})$. See the first image in Figure~\ref{fig:sector domains}.

By choosing $\delta$ large enough, and possibly shrinking $\varepsilon$, we get $|\zbar' - \zbar''| \geq |z_\ell' - z_\ell''| > 2\mu$, a contradiction.

If the above does not happen, we observe that $[\xbar,\ybar] \subseteq \Dte$ as soon as $\xbar$, $\ybar$ lie in a polydisc at $\cbar$ of slightly smaller radius; to be precise, as soon as $|\xbar|,|\ybar| \geq \sqrt{\varepsilon^{-2} + \mu^2}$. It then suffices to shrink $D^*$ a little further to reach the desired conclusion. See the second image in Figure~\ref{fig:sector domains}, where $\varepsilon_{\rm{old}}$ represents the starting value of $\varepsilon$, and $\varepsilon_{\rm{new}}$ the new one.
\end{proof}

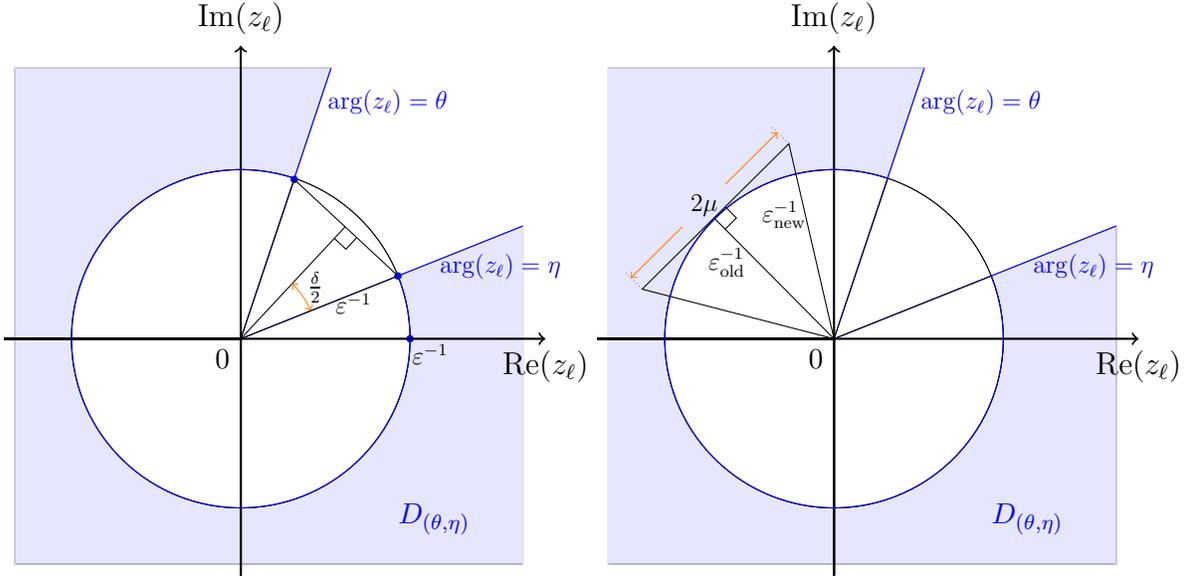
\begin{figure}[h]
  \centering
  \begin{minipage}{0.5\textwidth}
  \centering
      \begin{tikzpicture}[scale=1.5]
      
      \draw[fill=blue!10]  (0,0) -- (0.8,2.4)--(-2,2.4) -- (-2,-2) -- (2.5,-2) -- (2.5,1) -- cycle;
      \draw[blue!10]  (0,0) -- (0.8,2.4)--(-2,2.4) -- (-2,-2) -- (2.5,-2) -- (2.5,1) -- cycle;
          
      \draw (0,0) circle (1.5cm);

      \draw (0,0) [fill = white] circle (1.5cm);
      
      \draw[blue]  (0,0) -- (0.8,2.4); 
      \draw[blue]  (0,0) -- (2.5,1);
      
      \draw[blue] (1.3,-1.6) circle (0pt) node [right] {\small $\Dte$};
      \draw[blue] (0.6,2.1) circle (0pt) node [right] {\scalebox{0.8}{ $\arg(z_\ell) = \theta$}};
      \draw[blue] (2.3,0.85) circle (0pt) node [below] {\scalebox{0.8}{$\arg(z_\ell) = \eta$}};
      
      \draw (0,0) circle (0pt) node [below left] {\small $0$};
       \draw (1.35,0.05) circle (0pt) node [below right] {\scalebox{0.8}{ $\varepsilon^{-1}$}};
       \filldraw[blue] (1.5, 0) circle (0.8pt);
       
      \coordinate (A) at (0,0);
      \coordinate (B) at (0.474, 1.413);
      \coordinate (C) at (1.393, 0.557);
      \coordinate (P) at ($(B)!(A)!(C)$);
      
      \draw[blue, fill] (B) circle (0.8pt);
      
      \draw[blue, fill] (C) circle (0.8pt);
      
      \draw (A) -- (P);
      
      \draw (A) -- (B) -- (C) -- cycle;
      
     \dotMarkRightAngle[size=4pt](A,P,C);
     
     \pic["\scalebox{0.8}{$\frac {\delta}{2}$}", draw=orange, <->, angle eccentricity=1.2, angle radius=1cm]    {angle=C--A--P};

         \draw (1, 0.5) circle (0pt) node [below] {\scalebox{0.8}{$\varepsilon^{-1}$}};
          
       \draw [<->, thick] (0,2.6) node (yaxis) [above] {$\im (z_\ell)$} -- (0,-2.1) 
          |- (-2.1,0) -- (2.7,0) node (xaxis) [below] {$\re (z_\ell)$};
          
      \draw [blue,domain=71:382, samples =100] plot ({1.5*cos(\x)}, {1.5*sin(\x)});
  \end{tikzpicture}
  \end{minipage}
  \begin{minipage}{0.49\textwidth}
  \centering
      \begin{tikzpicture}[scale=1.5]
      
      \draw[fill=blue!10]  (0,0) -- (0.8,2.4)--(-2,2.4) -- (-2,-2) -- (2.5,-2) -- (2.5,1) -- cycle;
      \draw[blue!10]  (0,0) -- (0.8,2.4)--(-2,2.4) -- (-2,-2) -- (2.5,-2) -- (2.5,1) -- cycle;
      
      \coordinate (A') at (-1.8, 0.54);
      \coordinate (B') at (-0.5, 1.83);
      \draw[orange, <->] (A') -- (B') node [midway, fill = blue!10, text = black] {\scalebox{0.8}{$2\mu$}};
          
      \draw (0,0) circle (1.5cm);

      \draw (0,0) [fill = white] circle (1.5cm);
      
      \draw[blue]  (0,0) -- (0.8,2.4); 
      \draw[blue]  (0,0) -- (2.5,1);
      
      \draw[blue] (1.3,-1.6) circle (0pt) node [right] {\small $\Dte$};
      \draw[blue] (0.6,2.1) circle (0pt) node [right] {\scalebox{0.8}{ $\arg(z_\ell) = \theta$}};
      \draw[blue] (2.3,0.85) circle (0pt) node [below] {\scalebox{0.8}{$\arg(z_\ell) = \eta$}};
      
      \draw (0,0) circle (0pt) node [below left] {\small $0$};
      
      \coordinate (A) at (-1.7, 0.44);
      \coordinate (B) at (-0.4, 1.73);
      \coordinate (O) at (0,0);
      \coordinate (P) at ($(A)!(O)!(B)$);    
      \coordinate (M) at (0.474, 1.413);
      \coordinate (N) at (1.393, 0.557);
      \coordinate (A'') at (-1.85, 0.59);
      \coordinate (B'') at (-0.55, 1.88);
  
      \draw  (O) -- (A) -- (B) -- (O);
      \draw (O) -- (P);
      
      \draw (O) -- (M);
      \draw (O) -- (N);
      
      \draw[gray, densely dotted] (A) -- (A'');
      \draw[gray, densely dotted] (B) -- (B'');

      \draw (-0.95, 0.9) circle (0pt) node [below] {\scalebox{0.8}{$\varepsilon_{\mathrm{old}}^{-1}$}};
      
      \draw (-0.15, 1.1) circle (0pt) node [left] {\scalebox{0.8}{$\varepsilon_{\mathrm{new}}^{-1}$}};
  
     \dotMarkRightAngle[size=4pt](O,P,B);

       \draw [<->, thick] (0,2.6) node (yaxis) [above] {$\im (z_\ell)$} -- (0,-2.1) 
          |- (-2.1,0) -- (2.7,0) node (xaxis) [below] {$\re (z_\ell)$};
          
       \draw [blue,domain=71:382, samples =100] plot ({1.5*cos(\x)}, {1.5*sin(\x)});
          
  \end{tikzpicture}
  \end{minipage}
  \caption{Some distances between points in the sector domain $\Dte$.}
  \label{fig:sector domains}
\end{figure}

\paragraph{Step 4: Mapping the lattice to the solutions.}
In this step we describe the inverse of $F$, which maps lattice points to solutions of the equation $\exp_A(\zbar) = \alpha(\zbar)$. 

We fix a sector domain $\Dte^*$ with $\eta = \theta+2\pi-\delta$ as given by the previous step, so that $F$ is injective on that domain, and hence has an inverse. 
Now we show that the image of $F$ differs from its domain by at most a strip of bounded width around the boundary. 
Let $B\seq \C^n$ be a closed ball centred at $\mathbf{0}$ containing $G(\Dte^*)$. Then by the definition of $F$, for all $\zbar \in \Dte$ we have $F(\zbar) - \zbar \in B$.

\begin{prop}\label{prop:image-F abelian}
The image $E \leteq F(\Dte)$ is open and contains
\[E' \leteq \{ \zbar\in \Dte: \zbar+ {B} \subseteq \Dte\}.\]
\end{prop}

\begin{wrapfigure}{r}{8cm}
    \centering
    \begin{tikzpicture}[scale=1.5]

    \draw[fill=blue!10]  (0,0) -- (1.5,3)--(3.5,1) -- cycle;
    
    \draw[blue!10]  (1.5,3)--(3.5,1);
        
    \draw[fill=green!10, opacity=0.8]  (0.8,0.64) -- (1.9,2.84)--(3.36,1.38) -- cycle;

    \draw[green!10]  (1.9,2.84)--(3.36,1.38);
        
    \draw[blue] (0,0) -- (1.5,3);
    \draw[blue] (0,0) -- (3.5,1);
    \draw[green] (0.8,0.64) -- (1.9,2.84);
    \draw[green] (0.8,0.64) -- (3.36,1.38);
    \draw[green] (2.6,1.5) circle (0pt) node [right] {\small $E'$};

    \draw (1.25,1.54)[fill = gray, opacity = 0.2] circle (0.425cm);
    \draw (1.25,1.54) circle (0.425cm);
    \draw (1.56,1.83) circle (0pt) node [right] {\small $z{+}B$};
    
     \filldraw[red] (1.25,1.54) circle (1pt) node[black, anchor=west] {\small $z$};

     \filldraw[gray!80, domain=0:90, samples = 100] plot ({0.8*cos(\x)}, {0.8*sin(\x)});
     
     \draw [blue,domain=16:63, samples = 100] plot ({0.8*cos(\x)}, {0.8*sin(\x)});
     
     \filldraw [blue!10,domain=16:63, samples = 100] plot ({1.225*cos(\x)}, {1.225*sin(\x)});
     
      \filldraw[blue!10] (0.8,0.64) -- (1, 0.64) -- (0.8, 0.8) -- cycle;
     
     \filldraw[gray!50, opacity = 0.5, domain=0:90, samples = 100] plot ({1.225*cos(\x)}, {1.225*sin(\x)});
     
     \draw[green, domain=35:43, samples = 100] plot ({1.225*cos(\x)}, {1.225*sin(\x)});

     \filldraw[gray!80] (0,0) -- (0.8,0) -- (0, 0.8) -- cycle;

      \filldraw[gray!50, opacity = 0.5] (0,0) -- (1.22,0) -- (0, 1.22) -- cycle;
      
      \draw[purple,rounded corners=14pt] (1.5,3.2) -- (1.2,2) -- (0.8,1.3) -- (0.5,0.9);
      \draw[purple,rounded corners=12pt] (0.5,0.9) -- (0.75,0.8) -- (1,0.4);
      \draw[purple,rounded corners=14pt] (1,0.4) -- (1.7,0.6) -- (2.4,0.7) -- (3.3,1.1) -- (3.5,1.15);
      \draw[purple] (2.2,0.62) node [above] {\small $\partial E$};

      \draw[blue] (2.1, 0.2) circle (0pt) node [left] {\small $\Dte$};
    
    \draw[blue] (2.6,0.7) circle (0pt) node [right] {\scalebox{0.8}{ $\arg(z_\ell) = \theta$}};
    
    \draw[blue] (1.4,2.8) circle (0pt) node [left] {\scalebox{0.8}{ $\arg(z_\ell) = \eta$}};
     
    \draw [<->,thick] (0,3.2) node (yaxis) [above] {$\im (z_\ell)$}
        |- (4,0) node (xaxis) [below] {$\re (z_\ell)$};

    \end{tikzpicture}
\end{wrapfigure}

\noindent \textit{Proof.} 
Since $\dd F(\zbar)$ is non-singular on $\Dte$, by the inverse function theorem, $F$ is a local homeomorphism, hence an open map. So the image $E$ is open and connected in $\C^n$. (See figure on the right for a pictorial representation of how $\Dte$, $E'$ and $\partial E$ might look.)

Clearly, $E\cap E'$ is a non-empty open subset of $E'$. Hence if $E'\setminus E \neq \emptyset$ then $\partial E \cap E' \neq \emptyset$ where $\partial$ denotes the boundary. Take a point $\xbar \in \partial E \cap E'$ and a small closed neighbourhood $\xbar \in N\seq E'$. 
 Then we have \[  \overline{F^{-1}(N)} \seq \overline{N+B} = N+B  \seq \Dte.\] 
Now pick a sequence $\xbar_k\in N\cap E$ with $\xbar_k \to \xbar$ as $k\to \infty$. Since $F^{-1}(N)$ is a non-empty bounded subset of $\Dte$, if we set $\zbar_k := F^{-1}(\xbar_k)$, then a subsequence of $\zbar_k$ has a limit point $\zbar \in \overline{F^{-1}(N)}\seq \Dte$. Then by continuity of $F$ we conclude that $\xbar = F(\zbar)\in E$, which is a contradiction. 
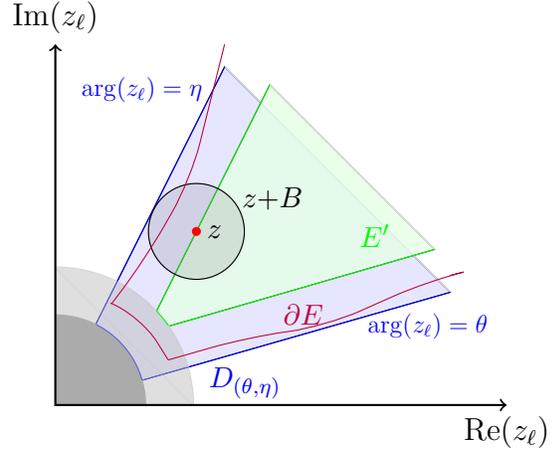
\hfill\qedsymbol

\begin{proposition}\label{prop:maps S}
For small enough $D^*$, there is a small $\delta'>0$ such that for all $\theta \in \R$ and $\eta = \theta+2\pi - \delta'$, there is a map $S: \Dte^* \to \P_n$  with the following properties:
\begin{enumerate}
 \item For all $\zbar \in \Dte^*$, we have  $F(S(\zbar)) = \zbar$, where we take the branch $F$ of $\mathbf{F}$ defined on $\Dte^*$ or its analytic continuation to a slightly larger domain containing the image of $S$.
 \item In particular, for $\zbar \in D^*\setminus D$ we have $S(\zbar) = \zbar$ and for $\zbar \in \Dte$ we have $S(\zbar) -\zbar \in B$, the bounded ball defined above.
 \item  For $\zbar \in \Dte$ we have $\exp_A(S(\zbar)) = \alpha(S(\zbar))$ if and only if $\zbar \in \Lambda$.
 \item  $S$ is continuous on $\Dte^*$ and holomorphic on $\Dte$;
 \item The restriction of $S$ to the finite part of the domain $\Dte$ is asymptotically a translation. More precisely, 
    \[ S(\zbar) - \zbar \to G(\cbar) \text{ as } \zbar \to \cbar \text{ with } \zbar \in \Dte.\]
 \end{enumerate}
\end{proposition}

\begin{proof}
For a suitable small $\delta'$, the set $E'$ from the previous proposition contains the sector domain $D_{(\theta+\delta'/2,\eta-\delta'/2)}$, except for a strip of bounded width near the part of the boundary given by $|z_\ell| = \varepsilon^{-1}$. 
We shrink $\varepsilon$ to remove this bounded strip, and then the image of $F$ (analytically continued from the new $\Dte^*$ back to the original domain with larger $\varepsilon$) contains $D_{(\theta+\delta'/2,\eta-\delta'/2)}^*$. 
Since $F$ is injective on that domain, we can define $S$ to be its set-theoretic inverse map with domain $D_{(\theta+\delta'/2,\eta-\delta'/2)}^*$. 
Relabelling $\theta+\delta'/2$ as $\theta$ and $\theta+2\pi-\delta'$ as $\eta$, we get the $S$ of the statement of the proposition satisfying point~1. 
Points~2 and~3 follow from the properties of $F$. 
It follows from the inverse function theorem that $S$ is holomorphic on $\Dte$. 
Continuity of $S$ on $\Dte^*$ follows the same way as continuity of $F$ using the fact that $S(\zbar) - \zbar$ is bounded.

For point~5, observe that $\zbar = F(S(\zbar)) = S(\zbar) - G(S(\zbar))$, and so $S(\zbar) - \zbar = G(S(\zbar))$. By continuity of $S$, we get that $S(\zbar) \to S(\cbar) = \cbar$ for $\zbar \to \cbar$, hence $S(\zbar) - \zbar = G(S(\zbar)) \to G(\cbar)$ for $\zbar \to \cbar$ by continuity of $G$.
\end{proof}

\paragraph{Step 5: Analytic continuation of the solution map \texorpdfstring{$S$}{S}.}

In this final step, we finish the proof of Theorem~\ref{thm-main-abelian} by considering the maps $S$ from the domains $\Dte^*$ as a sheaf giving a multivalued map. Just as we did for $\mathbf{F}$ and $\mathbf{G}$, the union of the maps $S$ yields a multivalued map $D^* \to \P_n$. We now push this further by allowing $\cbar$ to vary along the set $C \subseteq \P_n \setminus \C^n$ of Proposition~\ref{alg map prop}. We shall also restrict the domains from $D^* \subseteq \P_n$ to $D \subseteq \C^n$, in order to get \emph{analytic} maps, without the potential singularities at infinity which become irrelevant in our final conclusion.

For each $\cbar \in C$ we have an open polydisc $D^*$ around $\cbar$. Since $C$ is open in $H = \P_n \setminus \C^n$, we may always assume, after shrinking $D^*$, that $D^* \cap H$ is a subset of $C$. We now write this $D^*$ as $D^*_{\cbar}$ and define $\Omega^* = \bigcup_{\cbar\in C} D^*_{\cbar}$, an open subset of $\P_n$. Then $\Omega^*\setminus \C^n$ is indeed $C$ since obviously $\cbar \in D_{\cbar}^*$, so by construction $C \subseteq \Omega^* \cap H \subseteq C$.
 
For each $\cbar$ and for each interval $(\theta,\eta) \subs \R$ with $\eta \le \theta+2\pi - \delta'$ we have a map $S:D^*_{\cbar,(\theta,\eta)}\to \P_n$. It is clear that such maps, when restricted to $D_{\cbar,(\theta,\eta)}$ so to become analytic, are continuations of each other, in the following sense: for any two maps $S$, $\tilde{S}$ as above, the set $\{ S(\zbar) = \tilde{S}(\zbar) \}$ is both closed and open in the intersection of their domains. Thus, as in Remark~\ref{sheaf remark}, their restrictions to $\Omega = \Omega^* \cap \C^n$ generate a sheaf $\mathbf{S}$ of analytic maps, the union of which is an analytic multivalued map $\Omega \to \C^n$.

This gives us the data of Theorem~\ref{thm-main-abelian}. We can now prove that $\mathbf{S}$ has the required properties.

\begin{enumerate}
  \item Note that $\mathbf{S}$ is a local homeomorphism, because its local inverses are by Proposition~\ref{prop:Jac < epsilon}. Thus its image $\mathbf{S}(\Omega)$ is open.
  
  Fix a fundamental domain of $\C/\Lambda$, and let $\nu$ be its diameter. We claim that for every $\bm{a} \in \Omega$, there is a branch $S$ of $\mathbf{S}$ such that $|S(\abar) - \abar| < \mu + \nu$. Indeed, if we pick $\cbar$ and a sector domain of $D_{\cbar}$ containing $\bm{a}$, we can choose a branch $G$ of $\mathbf{G}$ on that sector domain with $|G(\cbar)| < \nu$, and find $|S(\abar) - \abar| = |G(S(\abar))| \leq |G(S(\abar)) - G(\cbar)| + |G(\cbar)| < \mu + \nu$ by Proposition~\ref{prop:bounded}. Moreover, we may assume that $|S(\zbar) - \zbar| < \mu + \nu$ for every $\zbar$ in a neighbourhood of $\abar$.
  
  Given this, it suffices to reason as in Proposition~\ref{prop:image-F abelian}. Suppose that $\Omega' = \{\zbar \in \Omega : \zbar + B \subseteq \Omega\}$ is not contained in $\Omega$, where $B$ is the closed ball at $\bm{0}$ of radius $\mu + \nu$. Then there exists $\abar \in \Omega'$ on the boundary of $\mathbf{S}(\Omega)$. Pick a local branch $S$ on a small neighbourhood of $\abar$, all contained in $\Omega'$, satisfying $|S(\zbar) - \zbar| < \mu + \nu$, and a sequence $\xbar_k$ in the image of $S$ such that $\xbar_k \to \abar$. Now observe that the preimages $S^{-1}(\xbar_k)$ must converge to some $\xbar \in \Omega$ such that $S(\xbar) = \abar$, a contradiction.

  It follows that every point of $\Omega$ not in $\mathbf{S}(\Omega)$ has distance at most $\mu + \nu$ from the boundary $\partial\Omega$ in $\C^n$.

  \item This follows from the 3\textsuperscript{rd} condition of Proposition~\ref{prop:maps S}, together with the fact that the branches of $\alpha$ cover $V$ (which is point~2 of Proposition~\ref{alg map prop}), with the same proviso as above about the strip of bounded width at the boundary.

  \item For each sector domain $\Dte$ of $\Omega$, there are $d$ distinct branches $G_1, \dots, G_d$ of $\mathbf{G}$ on $\Dte$, corresponding to the $d$ distinct branches of $\alpha$, such that every other branch is of the form $G_i + \lambar$ for some $\lambar \in \Lambda$. Thus the same is true for $\mathbf{F}$ for some branches $F_1, \dots, F_d$.

  Let $S_1, \dots, S_d$ be the corresponding branches obtained in Proposition~\ref{prop:maps S}. Then every branch of $\mathbf{S}$ is of the form $S_i(\zbar + \lambar)$ for some $1 \leq i \leq d$ and $\lambar \in \Lambda$.

  To conclude, we observe that such branches are all distinct. Suppose that $S_i(\zbar + \lambar) \equiv S_{i'}(\zbar + \lambar')$ on $\Dte$. Then also $F_i - \lambar \equiv F_{i'} - \lambar'$ on $\Dte$, as the branches of $\mathbf{F}$ are local inverses of the branches of $\mathbf{S}$. In turn, $i = i'$ and $\lambar = \lambar'$, as desired.

  \item This is the asymptotic condition from Proposition~\ref{prop:maps S}, with the points $\gambar$ being the values of $G(\cbar)$.

  \item It is now clear that the set $\mathbf{S}(\Omega \cap \Lambda)$ is Zariski dense in $\C^n$ and, since $\dim V = n$, it follows at once that $\{ (\zbar,\exp_A (\zbar))\in V \}$ is Zariski dense in $V$.
\end{enumerate}

That completes the proof of Theorem~\ref{thm-main-abelian}. \qed

\begin{remark} 
For each $\cbar$, the polydisc $D^*_{\cbar}$ is given with a radius $\varepsilon = \varepsilon_{\cbar}$, but we do not have any uniformity in $\varepsilon$ as $\cbar$ varies. Indeed, as $\cbar$ approaches the boundary of $C$, we may have $\varepsilon_{\cbar} \to 0$. Furthermore, the index $\ell$ could vary as well. For points $\abar \in \P_n \setminus (\C^n \cup C)$, the algebraic map $\alpha$ may have worse singularities than the ramification points we have dealt with. For these reasons, we do not have a complete description of all the exponential points $(\zbar,\exp_A(\zbar)) \in V$ which are large, that is, such that $|\zbar|$ is larger than some given $\varepsilon^{-1}$. However, each exponential point is known to be isolated, so they cannot accumulate anywhere in $\C^n$, so the points we have found should be a large proportion of the total in any meaningful sense.
\end{remark}

\begin{remark}\label{remark: general alpha}
  Our method can be used to prove the existence of solutions of any exponential equations of the form $\exp_A(\zbar) = \beta(\zbar)$, where $\beta : \Dte^* \to \C^n$ is holomorphic on $\Dte$ and continuous on $\Dte^*$, and $\theta, \eta$ are given. In steps 2--4, one can simply omit all references to the uniformity in $\theta, \eta$, and also replace $\theta, \eta$ with suitable values $\theta', \eta'$ satisfying $\theta < \theta' < \eta' < \eta$ when necessary.

  If one has a sufficiently rich understanding of the analytic continuations of $\beta$ around the points at infinity, the arguments of step 5 could be used to give a global description of the solutions in the style of Theorem~\ref{thm-main-abelian}.
\end{remark}

\section{The case of algebraic tori}\label{alg torus section}

\setcounter{assumption}{0}

Let $\exp: \C \to \gm$ be the usual exponential map. We will also let $\exp: \C^n \to \gm^n$ denote the exponential map of $\gm^n$ for any $n$, given by coordinate-wise action of the former map. The lattice of periods of $\exp$ is $\Lambda:= (2\pi i \Z)^n$.  We will write $\Log$ for the logarithmic map corresponding to $\exp$, and $\log : \R_{>0} \to \R$ for the real logarithm.

In this section we adapt the ideas of the previous section to prove a theorem for algebraic tori, analogous to Theorem~\ref{thm-main-abelian}. The only difference is the asymptotic behaviour of the solutions, which are no longer asymptotically translations of lattice points.

\begin{theorem}\label{thm-main-torus}
Let $V\seq \C^n \times \gm^n$ be a subvariety of dimension $n$ with dominant projection $\pi: V \to \C^n$. Let $d:= \deg \pi$. We embed $\C^n$ in projective space $\P_n$ in the usual way \eqref{Cn into Pn}.

Then there is a subset $\Omega^* \subs \P_n$, which is open in the complex topology, such that $C\leteq \Omega^* \setminus \C^n$ is Zariski open dense in $\P_n\setminus \C^n$, and there is a sheaf $\Sbar$ of analytic maps on $\Omega := \Omega^* \cap \C^n$ taking values in $\C^n$ with the following properties:

\begin{enumerate}
\item The image $\mathbf{S}(\Omega)$ contains $\Omega$ except possibly for a narrow strip along the boundary $\partial \Omega$.

\item For $\lambar \in \Omega \cap \Lambda$, each
value of $\Sbar(\lambar)$ satisfies $(\Sbar(\lambar),\exp(\Sbar(\lambar))) \in V$. Furthermore, these are the only exponential points $(\zbar, \exp(\zbar))$ of $V$ with  $\zbar \in \Omega$ (except possibly near the boundary).

\item These exponential points are locally in $d$-to-$1$ correspondence with the points of $\Lambda \cap \Omega$: $\mathbf{S}$ has $d$ branches, possibly up to translation of the argument by elements of $\Lambda$. 

\item 
The solutions $\Sbar(\lambar)$ are asymptotically close to lattice points: for each $\cbar \in C$ and each branch $S$ of $\mathbf{S}$ we have $S(\zbar) = \zbar + O(\log |\zbar|)$ for $\zbar \to \cbar$.

\item In particular, the set $\mathbf{S}(\Omega \cap \Lambda)$ is Zariski dense in $\C^n$, and the set of exponential points $\{ (\zbar,\exp (\zbar))\in V \}$ is Zariski dense in $V$.
\end{enumerate}

\end{theorem}

The proof follows that of Theorem \ref{thm-main-abelian} closely, so we will focus on the differences. The two essential differences are that abelian varieties are compact whereas  algebraic tori are not, and (relatedly) that the lattice $(2\pi i\Z)^n$ does not accumulate to every point at infinity in the complex topology (although it does in the Zariski topology). 
In the case of algebraic tori we have to deal with two extra points $0$ and $\infty$ (in dimension~1). Furthermore, any branch of the logarithmic map of an abelian variety is bounded, which is not true for algebraic tori but we are able to make do with logarithmic growth instead.

As in the abelian case, we split the proof into several steps.

\paragraph{Step 1: The algebraic maps.}

We embed $\gm$ into $\P_1$ identified with $\C\cup \{ \infty \}$, and consider the Zariski closure $\overline{V}$ of $V$ in $\P_n \cross \P_1^n$. We shall apply Proposition~\ref{alg map prop} and extract algebraic maps from the $\overline{V}$, but we want the images of those maps to be contained in $\gm^n$. To this end, let $Z:= \{ \zbar \in \C^n : (\zbar, 0)\in \overline{V} \text{ or } (\zbar, \infty)\in \overline{V} \}$. Then $Z$ has codimension~$\geq 1$ in $\C^n$, and the set $Z^*$ of its limit points in $H:=\P_n\setminus \C^n$ is a lower dimensional Zariski closed subset of $H$. So we shrink the set $C$ given by Proposition~\ref{alg map prop} by removing $Z^*$. To get our growth estimates later, we shrink $C$ further and assume that for any point $[0:t_1:\ldots:t_n]\in C$ none of the $t_i$ is $0$. One consequence is that we can work with the fixed chart $U_1$ rather than a varying chart $U_\ell$.

 Thus, we end up with the following data:

\begin{itemize}
    \item $C$, a Zariski open dense subset of $H$ such that for any point $[0:t_1:\ldots:t_n]\in C$ none of the $t_i$ is $0$;
    \item $\cbar:=[0:1:c_2:\ldots:c_n]$ an arbitrary point of $C$;
    \item $D^*$ a small polydisc at $\cbar$ in the chart ${U}_1$ as defined in \eqref{eq:polydisc}, chosen small enough that it does not meet the set $Z$ given above;
\end{itemize}

As in the abelian case, we will shrink $D^*$ (by reducing $\varepsilon$) to ensure certain properties of certain maps hold, which will be explicitly stated every time. In particular, we choose the $\varepsilon$ defining $D^*$ to be at most $\min \left\{\frac{|c_i|}2, \frac1{2|c_i|} : i=1,\ldots,n\right\}$ so that for any $\zbar = (z_1,\ldots,z_n)\in D$ we have $|z_1| > 2$, all other $|z_i| > 1$, and the coordinates of $\zbar$ are roughly proportional to each other:
\be\label{eq:roughly prop}
    \frac12 < \frac{|z_i|}{|c_i z_1|} < \frac32.
\ee
Of course, this assumption cannot be made uniformly as $\cbar$ varies, since we can have $|c_i|$ arbitrarily small.

For each $\theta \in \R$ and each $\eta \in (\theta,\theta+2\pi]$ we have the sector domains $\Dte \subs \C^n$ and their extensions $\Dte^*$ to $\P_n$ as given in \eqref{eq:sector domain} and \eqref{eq:sector domain 2}, and the algebraic maps
\begin{equation}
    \alpha^1, \dots, \alpha^d : \Dte^* \to (\P_1)^n. 
\end{equation}
As before, we drop the indices and write $\alpha$ to denote one of these maps.

\paragraph{Step 2: Mapping the solutions to the lattice.} 

Since $D \cap Z = \emptyset$, the restriction of $\alpha$ to $\Dte$ takes values in $\gm^n$. 
So, as in Section~\ref{abelian case section}, we  can choose a holomorphic branch $G$ of $\Log \circ\alpha$ on $D_{(\theta,\eta)}$:
\begin{gather}
      G : D_{(\theta,\eta)} \to \C^n \text{ holomorphic on } D_{(\theta, \eta)}, \nonumber \\ \exp(G(\zbar)) = \alpha(\zbar) \text{ for all } \zbar \in D_{(\theta,\eta)}.  \label{eq:G-alpha-torus}
\end{gather}
In this case, we cannot necessarily continue $G$ to a map $\Dte^*\to \C^n$, since as $\zbar \to \cbar \in C$ we may have some coordinate $\alpha_i(\zbar) \to 0 \text{ or } \infty$, where the logarithm is not defined.

We also remark that, as in the abelian case, the different choices of maps $G$ together yield a multivalued map $\Gbar : D \to \C^n$ with associated sheaf of analytic maps as in Remark~\ref{sheaf remark}.

By our choice of $\varepsilon$, all coordinates are roughly proportional on $D$ and are larger than $1$ in absolute value. Since $\alpha$ is an algebraic map, for each coordinate function $\alpha_i$ there is a positive integer $q_i\in \N$ such that for all $\zbar \in D$ we have
\begin{equation}\label{eq: alpha growth}
     |\zbar|^{-q_i} < |\alpha_i(\zbar)| < |\zbar|^{q_i}.
\end{equation}
Let $q := \max \{ q_i : 1\leq i \leq n \}$. Then for each coordinate-function $G_i$ of $G$, we have
\[ \left| \re(G_i(\zbar)) \right| = \left|\log |\alpha_i(\zbar)|\right| \leq q \log |\zbar|. \]

Similarly, the argument of $\alpha_i(\zbar)$ is bounded on the sector domain $\Dte$, hence so is the imaginary part of $G_i(\zbar)$.
So the ratio 
\[ \frac{G_i(\zbar)}{\log |\zbar|}  \]
is bounded. Thus, $|G(\zbar)| = O(\log |\zbar|)$ for $\zbar \in \Dte$, that is, $G$ has logarithmic growth as $|\zbar| \to \infty$.

Now define a map $F : \Dte^* \to \C^n$ by
\be
F(\zbar) \leteq 
\begin{cases}
\zbar - G(\zbar), &\text{ for } \zbar \in \Dte,\\
\zbar, &\text{ for } \zbar \in \Dte^*\setminus \Dte.
\end{cases}
\label{GM eq for F}
\ee

\begin{proposition}
    The map $F$ enjoys the following properties.
    
    \begin{itemize}
        \item  $F$ is continuous on $\Dte^*$ and holomorphic on $\Dte$;
        
        \item $F(\Dte)\seq \C^n$;
        
        \item A point $\zbar \in \Dte$ is a solution to the equation $\exp(\zbar) = \alpha(\zbar)$ if and only if $F(\zbar) \in \Lambda$.
    \end{itemize}
\end{proposition}
\begin{proof}
As in Proposition~\ref{prop F cont infty}, this is almost immediate except for showing that $F$ is continuous at any point  $\abar \in \Dte^* \setminus \Dte$. In the abelian case this followed since $G(\zbar)$ was bounded. Here we have $|G(\zbar)|$ growing logarithmically in $|\zbar|$, so again as we take the limit going to infinity in projective space, it becomes negligible compared to $|\zbar|$.
\end{proof}

As for $\mathbf{G}$, the maps $F$ and their continuations yield a sheaf of analytic maps from $D$ to $\C^n$, which we write as $\mathbf{F}(\zbar) = \zbar - \mathbf{G}(\zbar)$.

\paragraph{Step 3: Local injectivity of \texorpdfstring{$F$}{F}.}

Now we want to show that $F$ is injective on suitable domains $\Dte$. As in the abelian case, we need to estimate the partial derivatives of $G$.

\begin{lemma}\label{lem:log deriv}
There is $K >0$ such that, for any $\wbar \in \Dte$ and any $r>0$ such that the closed polydisc $T$ of radius $r$ at  $\wbar$ is contained in $\Dte$, we have
\be \| \dd G(\wbar) \| \leq K\frac{\log |\wbar|}{r}.\label{eq-log-growth} \ee
\end{lemma}

\begin{proof}
It suffices to establish such estimates for all partial derivatives of all coordinates of $G$. Let $g$ be a coordinate function of $G$. Then $g$ has logarithmic growth in $\zbar$, so applying the Cauchy Estimate (Fact \ref{lem:bounded-deriv}) yields 
\be\label{eq:alpha-der-max}
\frac{\partial g(\wbar)}{\partial z_k} = O\left(\frac{\max_{\zbar \in T}\log |\zbar|}{r}\right) .
\ee

If $T\seq \Dte$ is a polydisc of radius $r$, since $|z_m|\geq 1$ on $\Dte$ for every $m$, we must have $r< |w_m|$ for every $m$, hence $r< |\wbar| $. Therefore, if $\zbar \in T$ then $|\zbar| \leq |\wbar| + |\zbar - \wbar| \leq |\wbar| + r < 2 |\wbar|$. This then implies that $\log|\zbar| = O(\log|\wbar|)$ for $\zbar \in T$. So the desired bound follows from \eqref{eq:alpha-der-max}.
\end{proof}

We can now prove the analogue of Proposition \ref{prop:Jac < epsilon}, with a similar proof.
\begin{proposition}\label{prop:Jac < epsilon torus}
   For all $\nu > 0$, there is $D^*$ small enough such that for all intervals $(\theta,\eta)$, for all $\zbar \in \Dte$, the norm of the Jacobian matrix of the first partial derivatives $\dd G(\zbar)$ is less than $\nu$.
   
   In particular, for small enough $D^*$, for all branches $F$, $G$ of $\Fbar$, $\Gbar$ respectively on every sector domain $\Dte^*$, for all $\zbar \in \Dte$ we have:
   \begin{itemize}
       \item $\|\dd G(\zbar)\| < 1/2$,
       \item $\det(\dd F(\zbar)) \neq 0$.
   \end{itemize}
\end{proposition}
\begin{proof}
 Let $K$ be as in Lemma \ref{lem:log deriv}, and let $\zbar \in D$. If the polydisc of radius $r > \nu^{-1} K \log |\zbar|$ around $\zbar$ is entirely contained in $D$ then it is contained in  $D_{(\arg(z_1)-\pi,\arg(z_1)+\pi)}$, and then by Lemma \ref{lem:log deriv} we get $\| \dd G(\zbar) \| < \nu$. 
 Otherwise $|\zbar|$ is small, and we can shrink $D^*$ to remove this case. The rest follows easily.
\end{proof}

\begin{prop}\label{prop:injective-torus}
For small enough $D^*$, there is a small $\delta>0$ such that the map $F$ with domain $D^*_{(\theta,\theta + 2\pi - \delta)}$ is injective.
\end{prop}
\begin{proof}
The argument of Proposition \ref{prop:injective} goes through, except that we need to show that if $D^*$ is sufficiently small and $\xbar, \ybar \in D$ with $\arg(x_1) - \arg(y_1)=\delta$ then $|\xbar - \ybar| > |G(\xbar) - G(\ybar)|$. This can easily be deduced from the observation that $|\xbar - \ybar|$ is bounded below linearly in $|x_1|$ as $\xbar$ and $\ybar$ approach $\cbar$, while $|G(\xbar) - G(\ybar)|$ grows logarithmically in $|x_1|$.
\end{proof}

\paragraph{Step 4: Mapping the lattice to the solutions.}
We fix a small $\delta$ as in Proposition~\ref{prop:injective-torus}.

\begin{prop}\label{prop:image-F torus}
For small enough $D^*$ there is a small $\delta'$ such that for any $\theta \in \R$, writing $\eta = \theta+2\pi - \delta$, the image $F(\Dte^*)$ contains $D_{(\theta+\delta'/2,\eta-\delta'/2)}^*$, except for a strip of bounded width near the part of the boundary given by $|z_1| = \varepsilon^{-1}$. 
\end{prop}
\begin{proof}
The proof is as in Proposition~\ref{prop:image-F abelian}, except we use the logarithmic growth of $G$ in place of boundedness. More precisely, one should use $E' = \{ \zbar \in \Dte : B_{\zbar} \subseteq \Dte\}$, where $B_{\zbar}$ is the closed ball centred at $\zbar$ of radius $\mu\left|\log|\zbar|\right|$, and $\mu$ is chosen such that $|S(\zbar) - \zbar| < \mu \left|\log|z|\right|$ on $\Dte$. Then the argument of Proposition~\ref{prop:image-F abelian} shows that $F(\Dte^*)$ contains $E'$. It is then clear that $D_{(\theta+\delta'/2,\eta-\delta'/2)}^*$ is contained in $E'$, except for a strip of bounded width near the part of the boundary given by $|z_1| = \varepsilon^{-1}$.
\end{proof}

As before, we find a sheaf of maps $S$ which are local inverses to $F$. 
\begin{proposition}\label{prop:maps S torus}
For small enough $D^*$, for all $\theta \in \R$ and $\eta = \theta+2\pi - \delta- \delta'$, there is a map $S: \Dte^* \to \P_n$  with the following properties:
\begin{enumerate}
 \item For all $\zbar \in \Dte^*$, we have  $F(S(\zbar)) = \zbar$, where we take the branch of $F$ defined on $\Dte^*$ or its analytic continuation to a slightly larger domain. 
 \item In particular, for $\zbar \in D^*\setminus D$ we have $S(\zbar) = \zbar$.
 \item  For $\zbar \in \Dte$ we have $\exp(S(\zbar)) = \alpha(S(\zbar))$ if and only if $\zbar \in \Lambda$.
 \item  $S$ is continuous on $\Dte^*$ and holomorphic on $\Dte$.
 \item The restriction of $S$ to the finite part of the domain $\Dte$ is approximately given by $S(\zbar) \approx \zbar + G(\zbar)$. More precisely, 
    \[ S(\zbar) - \zbar -  G(\zbar) \to \mathbf0 \text{ as } \zbar \to \cbar \text{ with } \zbar \in \Dte.\]
     \end{enumerate}
\end{proposition}

\begin{proof}
The proof of points 1--4 is the same as in the abelian case. For point~5 observe, as in the abelian case, that $S(\zbar) -\zbar = G(S(\zbar))$ and
\[
|(S(\zbar) - \zbar) - G(\zbar)| = |G(S(\zbar)) - G(\zbar)| \leq \max_{\wbar\in [\zbar,S(\zbar)]}\|\dd G(\wbar)\| \cdot |S(\zbar)-\zbar|.\]
Since  $\|\dd G(\wbar)\| \to 0$ by Proposition \ref{prop:Jac < epsilon torus}, we have $|G(S(\zbar)) - G(\zbar)| = o(|G(S(\zbar))|)$ as $\zbar \to \cbar$. By the triangle inequality $||G(S(\zbar))| - |G(\zbar)|| = o(|G(S(\zbar))|)$ which implies that $\frac{|G(S(\zbar))|}{|G(\zbar)|}\to 1$ and so $|S(\zbar) - \zbar| \sim |G(\zbar)| = O(\log |\zbar|)$. 

Further, there is a constant $\gamma > 0$ such that for all sufficiently large $\zbar \in \Dte$ there are $\theta', \eta'$ such that the polydisc of radius $\gamma |\zbar|$, centred at $\zbar$, is contained in $D_{(\theta',\eta')}$. Then by Lemma~\ref{lem:log deriv} we have $\|\dd G(\zbar)\| \leq \gamma^{-1}K \frac{\log |\zbar|}{|\zbar|}$. When $\zbar$ is sufficiently large, so is $\wbar \in [\zbar, S(\zbar)]$, hence $$\max_{\wbar\in [\zbar,S(\zbar)]}\|\dd G(\wbar)\| = O\left( \frac{\log |\zbar|}{|\zbar|} \right)$$ and so 
\[\max_{\wbar\in [\zbar,S(\zbar)]}\|\dd G(\wbar)\| \cdot |S(\zbar)-\zbar| =  O\left( \frac{(\log |\zbar|)^2}{|\zbar|} \right)\]
so it tends to $0$.
Thus, $|(S(\zbar) - \zbar) - G(\zbar)| \to 0$ as $\zbar \to \cbar$. 
\end{proof}

\paragraph{Step 5: Analytic continuation of \texorpdfstring{$S$}{S}.}
The patching together of the maps $S$ on the domains $\Dte^*$, yielding the desired sheaf $\Sbar$, is done exactly as in the abelian case.

Point~1 of Theorem~\ref{thm-main-torus} can also be proven as in the abelian case, with the following changes. One observes that for every point $\abar \in \Omega$, there is a branch $S$ of $\Sbar$ around $\abar$ satisfying the inequality $|S(\zbar) - \zbar| < \mu \left|\log|\zbar|\right|$ for some $\mu$, by the estimates of Step~2. Furthermore, one can choose the branches $S$ so that the imaginary part of $S(\abar)$ is bounded uniformly in $\abar$. We can then assume that $\mu$ is chosen uniformly in $\abar$. It now suffices to take $\{\zbar \in \Omega : B_{\zbar} \subseteq \Omega\}$ as the set $\Omega'$, where $B_{\zbar}$ is the closed ball centred at $\zbar$ of radius $\mu \left|\log|\zbar|\right|$ as in Proposition~\ref{prop:image-F torus}, and continue the proof as in the abelian case. One finds that the image $\mathbf{S}(\Omega)$ covers $\Omega$ except possibly for points $\zbar$ of distance at most $\mu \left|\log|\zbar|\right|$ from the boundary.

Point~2 follows from point~3 of Proposition~\ref{prop:maps S torus}. 

Point~3 can be proven as in the abelian case by working with $d$ branches of $\Fbar$ and $\Gbar$.

Point~4 follows from Proposition~\ref{prop:maps S torus} (5).

For point~5, observe that $\Omega \cap \Lambda \neq \emptyset$ when all the ratios $\frac{c_k}{c_j}$ are rational. Moreover, even though the lattice $\Lambda$ has rank $n$ (as opposed to $2n$ in the abelian case), the set $\Omega \cap \Lambda$ is Zariski dense in $\C^n$, and so is $\Sbar(\Omega \cap \Lambda)$. That completes the proof of Theorem~\ref{thm-main-torus}, and Theorem~\ref{algebraic torus case} follows.\qed

\begin{remark}
  As in the abelian case (Remark~\ref{remark: general alpha}), our method yields solutions of $\exp(\zbar) = \beta(\zbar)$ for any holomorphic map $\beta : \Dte \to \C^n$, provided we have control on its growth rate as in \eqref{eq: alpha growth}. In particular, this gives an alternative proof of \cite[Rem.\ 2.8]{daquino--fornasiero-terzo}.
\end{remark}

\section{Final remarks}

Our methods can be adapted to solve more general systems of exponential equations, for example those combining exponential and $\wp$-functions. We indicate below that the analogue of Theorem \ref{dominant abelian case} holds for split semiabelian varieties, that is, complex algebraic groups that are isomorphic to a product of an abelian variety and an algebraic torus.

\begin{theorem}\label{split semiabelian case}
Let $S = A \cross \gm^q$ be a complex split semiabelian variety of dimension $n$. Let $V \subs \C^n \cross S$ be an algebraic subvariety with dominant projection to $\C^n$. Then there is $\zbar \in \C^n$ such that $(\zbar, \exp_{S}(\zbar)) \in V$. 
\end{theorem}

The proof is simply a combination of the proofs of Theorems \ref{thm-main-abelian} and \ref{thm-main-torus}, so we just present a brief outline.

\begin{proof}[Proof sketch]
Let $p=\dim A$ so that $p+q=n$. The exponential map of $S$ then can be written as $\exp_S = (\exp_A,\exp_{\gm^q}).$ If $\Lambda\seq \C^p$ is the lattice of periods of $\exp_A$ then $\Lambda\cross (2\pi i \Z)^q$ is the lattice of periods of $\exp_S$.

We extract a multivalued algebraic map $\alpha: D^* \to S$ from $V$ as before, and we can write $\alpha(\zbar)$ as $(\beta(\zbar),\gamma(\zbar))$ where $\beta : D^* \to A, \gamma : D^* \to \gm^q$.

Then locally on sector domains we can define
\[ F(\zbar) := \zbar - \Log_S \alpha(\zbar) = \zbar - (\Log_A \beta(\zbar), \Log_{\gm^q} \gamma(\zbar)).\]

The previous arguments now show that the total derivative of $F$ tends to the identity as $|\zbar|\to \infty$, and the rest of the proof follows as before.
\end{proof}

It seems likely that this method can be adapted to work for any semiabelian variety, without the split assumption. The issue in the semiabelian case is that we have neither the compactness of abelian varieties nor the explicit formulas for logarithmic maps as in the case of tori. So a better geometric or analytic understanding of the logarithmic maps would be needed. For instance, one needs an appropriate estimate on the growth of the semiabelian logarithm, so as to ensure that the total derivative of $F$ tends to the identity.

Relaxing the assumption on $V$ that the projection to $\C^n$ is dominant seems more difficult. Our proof ultimately depends on the same good asymptotic behaviour of a suitable function $F$ as $|\zbar| \to \infty$ in $\C^n$ as was used in \cite{brown-masser} to apply Newton's method. When $V$ does not project dominantly to $\C^n$, any analogous function $F$ we define can oscillate or grow too fast. Nonetheless, we hope that our approach using geometric and topological methods to show that the image of $F$ contains lattice points will be more robust to such issues than Newton's method is.

\bibliographystyle{alpha}
\bibliography{ref}

\end{document}